\documentclass[a4paper, 12pt, reqno]{amsart}

\usepackage{amssymb}
\usepackage[backref=page]{hyperref}
\usepackage[margin=0.75in]{geometry}
\usepackage{enumitem}
\usepackage{tikz-cd}
\usepackage{zref-clever}

\zcsetup{cap, nameinlink}

\newtheorem{theorem}{Theorem}[section]
\newtheorem{lemma}[theorem]{Lemma}
\AddToHook{env/lemma/begin}{\zcsetup{countertype={theorem=lemma}}}
\newtheorem{proposition}[theorem]{Proposition}
\AddToHook{env/proposition/begin}{\zcsetup{countertype={theorem=proposition}}}
\newtheorem{corollary}[theorem]{Corollary}
\AddToHook{env/corollary/begin}{\zcsetup{countertype={theorem=corollary}}}

\theoremstyle{remark}
\newtheorem{remark}[theorem]{Remark}
\AddToHook{env/remark/begin}{\zcsetup{countertype={theorem=remark}}}
\newtheorem{example}[theorem]{Example}
\AddToHook{env/example/begin}{\zcsetup{countertype={theorem=example}}}

\zcRefTypeSetup{section}{
  Name-sg=\S,
  name-sg=\S,
  Name-pl=\S\S,
  name-pl=\S\S,
  namesep={}
}

\setlist[enumerate, 1]{label=\normalfont(\roman*)}

\DeclareMathOperator{\Ind}{Ind}
\DeclareMathOperator{\ch}{ch}
\DeclareMathOperator{\len}{len}
\DeclareMathOperator{\clen}{clen}
\DeclareMathOperator{\gr}{gr}
\DeclareMathOperator{\Hom}{Hom}
\DeclareMathOperator{\vspan}{span}
\DeclareMathOperator{\Frac}{Frac}
\DeclareMathOperator{\lp}{lp}
\DeclareMathOperator{\prt}{prt}
\DeclareMathOperator{\Der}{Der}
\DeclareMathOperator{\End}{End}
\DeclareMathOperator{\inc}{inc}

\newcommand{\vac}{{|0\rangle}}
\newcommand{\Vir}{{\mathrm{Vir}}}
\newcommand{\Id}{{\mathrm{Id}}}
\newcommand{\psn}{{\mathrm{psn}}}

\renewcommand*{\backref}[1]{}
\renewcommand*{\backrefalt}[4]{%
  \ifcase #1 (Not cited.)%
  \or        (Cited on page~#2.)%
  \else      (Cited on pages~#2.)%
  \fi}

\begin{document}

\setcounter{section}{-1}

\title{Boundary minimal models and the Rogers-Ramanujan identities}
\author{Diego Salazar}
\thanks{The author is partially supported by FAPERJ grant 201.445/2021 and CNPq grant 152005/2025-0}
\address{Universidade Federal do Rio de Janeiro, Macaé, RJ, Brazil}
\email{xdiego.sg@gmail.com}
\date{\today}

\begin{abstract}
  We determine when the irreducible modules $L(c_{p, q}, h_{m, n})$ over the simple Virasoro vertex algebras $\Vir_{p, q}$, where $p, q \ge 2$ are relatively prime with $0 < m < p$ and $0 < n < q$, are classically free.
  It turns out that this only happens with the boundary minimal models, i.e., with the irreducible modules over $\Vir_{2, 2s + 1}$ for $s \in \mathbb{Z}_+$.
  We thus obtain a complete description of the classical limits of these modules in terms of the jet algebra of the corresponding Zhu $C_2$-algebra.
  The Andrews-Gordon generalization of the Rogers-Ramanujan identities is used in the proof, and our results in turn provide a natural interpretation of these identities.

  \noindent \textbf{Keywords.} Vertex algebras, boundary minimal models, Rogers-Ramanujan identities, quantum algebra, combinatorics.
\end{abstract}

\maketitle

\numberwithin{equation}{section}

\section{Introduction}
\label{sec:introduction}

In \cite{li_abelianizing_2005}, Li introduced a decreasing filtration $(F_pV)_{p \in \mathbb{Z}}$ on an arbitrary vertex algebra $V$.
The associated graded space $\gr_F(V)$ with respect to this decreasing filtration carries the structure of a vertex Poisson algebra.
Li also introduced a decreasing filtration $(F_pM)_{p \in \mathbb{Z}}$ for modules $M$ over a vertex algebra $V$ and showed that the associated graded space $\gr_F(M)$ is a module over the vertex Poisson algebra $\gr_F(V)$.

Given a commutative algebra $R$, the jet algebra $JR$ is constructed as the smallest differential algebra containing $R$.
It was proved by Arakawa in \cite[Proposition 2.5.1]{arakawa_remark_2012} that there exists a natural surjection
\begin{equation*}
  JR_V \twoheadrightarrow \gr_F(V)
\end{equation*}
of vertex Poisson algebras, where $R_V$ is the Zhu $C_2$-algebra of $V$ as defined in \cite{zhu_modular_1996}.
When this is an isomorphism, we say $V$ is classically free.
For modules $M$ over the vertex algebra $V$, we also have a natural surjection
\begin{equation*}
  JR_V \otimes_{R_V} R_M \twoheadrightarrow \gr_F(M),
\end{equation*}
and we say $M$ is classically free if this is an isomorphism.
One may in turn regard the classically free vertex algebra $V$ or $V$-module $M$ as a quantization of $JR_V$ or $JR_V \otimes_{R_V} R_M$, respectively.

The question of classical freeness of the simple Virasoro vertex algebra $\Vir_{p, q}$ was studied in \cite{van_ekeren_chiral_2021}, and its answer is very simple.
If $q > p \ge 2$ are relatively prime integers, then $\Vir_{p, q}$ is classically free if and only if $p = 2$.
The main objective of this article is to solve this problem for all irreducible modules over the vertex algebras $\Vir_{p, q}$, namely $L(c_{p, q}, h_{m, n})$ for integers $m, n$ such that $0 < m < p$ and $0 < n < q$.

It turns out that this problem is deeply connected with the Rogers-Ramanujan $q$-series identities and the Andrews-Gordon generalization \cite{gordon_combinatorial_1961} and \cite{andrews_analytic_1974}, namely
\begin{equation}
  \label{eq:1}
  \sum_{k = (k_1, \dots, k_{s - 1}) \in \mathbb{N}^{s - 1}}\frac{q^{\frac{1}{2}kG^{(s)}k^\top + kB^{(s)}_{s - i}}}{(q)_{k_1}\dots(q)_{k_{s - 1}}} = \prod^\infty_{\substack{n = 1 \\ n \not\equiv 0, \pm i \bmod 2s + 1}}(1 - q^n)^{-1},
\end{equation}
for $s \in \mathbb{Z}_+$, $i \in \{1, \dots, s\}$, where
\begin{align*}
  G^{(s)} &= (2\min\{i, j\})_{i, j = 1}^{s - 1}, \\
  B^{(s)}_j &= (0, 0, \dots, 0, 1, 2, \dots, j)^\top \quad \text{for $j = 0, 1, \dots, s - 1$},
\end{align*}
and ${}^\top$ denotes the transpose matrix.

The identity \eqref{eq:1} has an interpretation in terms of integer partitions, as explained in \cite[\S8]{andrews_theory_1998}.
We represent a partition $\lambda$ either as a finite nonincreasing sequence of positive integers $[\lambda_1, \dots, \lambda_s]$ or as an infinite sequence $(f_n)_{n \in \mathbb{Z}_+} = [1^{f_1}, 2^{f_2}, \dots]$ of nonnegative integers in which only finitely many of the $f_n$ are nonzero, where $f_n$ is the number of times $n$ appears among the $\lambda_i$.
We have a natural partial order $\le$ on partitions given by $(g_n)_{n \in \mathbb{Z}_+} \le (f_n)_{n \in \mathbb{Z}_+}$ if $g_n \le f_n$ for $n \in \mathbb{Z}_+$.
A set $P$ of partitions is said to be a partition ideal if whenever $(f_n)_{n \in \mathbb{Z}_+} \in P$ and $(g_n)_{n \in \mathbb{Z}_+} \le (f_n)_{n \in \mathbb{Z}_+}$, then $(g_n)_{n \in \mathbb{Z}_+} \in P$.
We denote by $p(P, n)$ the number of partitions of $n$ that belong to $P$.
We say two partition ideals $P_1$ and $P_2$ are equivalent, denoted by $P_1 \sim P_2$, if $p(P_1, n) = p(P_2, n)$ for $n \in \mathbb{N}$.

Let $s \in \mathbb{Z}_+$, and let $i \in \{1, \dots, s\}$.
Gordon's generalization of the Rogers-Ramanujan identities asserts
\begin{equation*}
  A^{s, i} \sim P^{s, i},
\end{equation*}
where the partition ideals $A^{s, i}$ and $P^{s, i}$ are defined by:
\begin{align*}
  A^{s, i} &= \{\text{$(f_n)_{n \in \mathbb{Z}_+}$ is a partition} \mid \text{for $n \in \mathbb{Z}_+$, $f_n > 0$ implies $n \not\equiv 0, \pm i \bmod 2s + 1$}\}, \\
  R^{s, i} &= \{\text{$\lambda$ is a partition} \mid \text{$\lambda = [\lambda_1, \dots, \lambda_s]$ satisfies $\lambda_1 - \lambda_s \le 1$ or $\lambda = [1^i]$}\}, \\
  P^{s, i} &= \text{the set of partitions that do not contain any partition in $R^{s, i}$}.
\end{align*}

As explained in \cite[\S8]{andrews_theory_1998}, $A^{s, i}$ is a partition ideal of order $1$, while $P^{s, i}$ is a partition ideal of order $2$.
It is an open problem to characterize the equivalence classes of partition ideals.
For partition ideals of order $1$, this problem is completely solved (see \cite[Theorem 8.4]{andrews_theory_1998}).

In the context of representations of infinite-dimensional Lie algebras and vertex algebras, the left-hand side of \eqref{eq:1} is precisely the graded dimension, or character, of the irreducible $\Vir$-module $M = L(c_{2, 2s + 1}, h_{1, i})$.
Our main theorems concern a refined version of this character and the intimate relation between the structure of $\gr_F(M)$ and the partition ideals $P^{s, i}$.
The second theorem about a PBW basis of $L(c_{2, 2s + 1}, h_{1, i})$ was conjectured by Kac in \cite{kac_modular_1988} and proved using different methods in \cite{feigin_annihilating_1992} and \cite{feigin_coinvariants_1993}.
Here, we present another proof relying on the refined character formula of the next theorem.

\begin{theorem}
  \label{thr:1}
  The refined character of $L(c_{2, 2s + 1}, h_{1, i})$ is given by
  \begin{equation*}
    \ch_{L(c_{2, 2s + 1}, h_{1, i})}(t, q) = q^{h_{1, i}}\left(\sum_{k = (k_1, \dots, k_{s - 1}) \in \mathbb{N}^{s - 1}}t^{kB^{(s)}_{s - 1}}\frac{q^{\frac{1}{2}kG^{(s)}k^\top + kB^{(s)}_{s - i}}}{(q)_{k_1}\dots(q)_{k_{s - 1}}}\right).
  \end{equation*}
\end{theorem}

\begin{theorem}
  \label{thr:2}
  The set
  \begin{equation*}
    \{L_{-\lambda_1}L_{-\lambda_2}\dots L_{-\lambda_m}|c_{2, 2s + 1}, h_{1, i}\rangle \mid \lambda = [\lambda_1, \dots, \lambda_m] \in P^{s, i}\}
  \end{equation*}
  is a vector space basis of $L(c_{2, 2s + 1}, h_{1, i})$.
\end{theorem}

\begin{theorem}
  \label{thr:3}
  The boundary minimal models $L(c_{2, 2s + 1}, h_{1, i})$ are classically free as modules over the simple Virasoro vertex algebra $\Vir_{2, 2s + 1}$ for $s \in \mathbb{Z}_+$ and $i \in \{1, \dots, s\}$.
  If $p, q > 2$ then no irreducible $\Vir_{p, q}$-module is classically free.
\end{theorem}
As a corollary, we obtain an explicit expression for the modules $\gr_F(L(c_{2, 2s + 1}, h_{1, i}))$ over the vertex Poisson algebra $\gr_F(\Vir_{2, 2s + 1})$.

The Andrews-Gordon series in \zcref{thr:1} already appeared in earlier works on vertex algebras such as \cite{feigin_quasi-particles_1993}, \cite{capparelli_rogersselberg_2006} and \cite{calinescu_vertex-algebraic_2008}, where they arise as characters of level $k = s - 1 \in \mathbb{N}$ principal subspaces of $\widehat{\mathfrak{sl}_2}$, and where related commutative algebraic structures capturing the Andrews-Gordon sum side were developed.

This article is organized as follows.
In \zcref{sec:prel-notat}, we review Verma modules and set notations related to partitions.
In \zcref{sec:modul-over-algebr}, we review certain filtrations and characters of algebras and modules, the main example being the PBW filtration of a Lie algebra.
In \zcref{sec:fields-over-vector}, we introduce an explicit formula related to the normal ordered product of a finite number of fields.
In \zcref{sec:modules-over-simple}, we recall the irreducible modules over the simple Virasoro vertex algebras.
In \zcref{sec:pbw-bases-characters}, we prove \zcref{thr:1} and \zcref{thr:2}.
In \zcref{sec:li-filtr}, we summarize the theory of filtrations of vertex algebras and their modules, with emphasis on explicit examples related to modules over the Virasoro Lie algebra.
In \zcref{sec:class-free-bound}, we recall Zhu $C_2$-algebra and the notion of classical freeness of vertex algebras and their modules.
Finally, we prove \zcref{thr:3}.
In the appendices, we recall notions related to Poisson algebras, jet algebras and vertex Poisson algebras.

Mathematical terms are typeset in \emph{italics} when they are officially defined.

\paragraph{Acknowledgements.}
I would like to thank Reimundo Heluani and Jethro Van Ekeren for their valuable suggestions.

\section{Preliminaries and notation}
\label{sec:prel-notat}

All vector spaces and all algebras are over $\mathbb{C}$, the field of complex numbers, unless otherwise stated.
The set of natural numbers $\{0, 1, \dots\}$ is denoted by $\mathbb{N}$, the set of integers is denoted by $\mathbb{Z}$, and the set of positive integers $\{1, 2, \dots\}$ is denoted by $\mathbb{Z}_+$.

First, we review the theory of representations of the Virasoro Lie algebra following \cite[\S3]{kac_bombay_2013}.
The \emph{Virasoro Lie algebra} is a Lie algebra given by
\begin{equation*}
  \Vir = \bigoplus_{n \in \mathbb{Z}}\mathbb{C}L_n \oplus \mathbb{C}C.
\end{equation*}
These elements satisfy the following commutation relations:
\begin{align*}
  [L_m, L_n] &= (m - n)L_{m + n} + \delta_{m, -n}\frac{m^3 - m}{12}C \quad \text{for $m, n \in \mathbb{Z}$}, \\
  [\Vir, C] &= 0.
\end{align*}
Let $(c, h) \in \mathbb{C}^2$.
We set $\Vir^\ge = \bigoplus_{n \in \mathbb{N}}\mathbb{C}L_n \oplus \mathbb{C}C$.
We make the subalgebra $\Vir^\ge$ of $\Vir$ act on $\mathbb{C}$ as follows:
\begin{equation*}
  \text{$L_n1 = 0$ for $n \in \mathbb{Z}_+$, $L_01 = h$ and $C1 = c$}.
\end{equation*}
The \emph{Verma representation} of $\Vir$ with \emph{highest weight} $(c, h)$ is defined as
\begin{equation*}
  M(c, h) = \Ind^\Vir_{\Vir^\ge}(\mathbb{C}) = U(\Vir) \otimes_{U(\Vir^\ge)} \mathbb{C},
\end{equation*}
where $\Vir$ acts by left multiplication.
We take $|c, h\rangle = 1\otimes1$ as \emph{highest weight vector}.

A \emph{composition (of $n \in \mathbb{N}$)} is a sequence $\lambda = [\lambda_1, \dots, \lambda_m]$ such that $\lambda_i \in \mathbb{Z}_+$ for $i = 1, \dots, m$ (and $\lambda_1 + \dots + \lambda_m = n$).
We also consider the \emph{empty composition} $\emptyset$, which is the unique composition of $0$.
Given a composition $\lambda = [\lambda_1, \dots, \lambda_m]$ and a permutation $\sigma \in S_m$, we define the composition
\begin{equation*}
  \lambda\sigma = [\lambda_{\sigma(1)}, \dots, \lambda_{\sigma(m)}].
\end{equation*}

For a composition $\lambda = [\lambda_1, \dots, \lambda_m]$, we define
\begin{equation*}
  L_\lambda = L_{-\lambda_1}\dots L_{-\lambda_m} \in U(\Vir),
\end{equation*}
the \emph{length of $\lambda$} as
\begin{equation*}
  \len(\lambda) = m,
\end{equation*}
the \emph{weight of $\lambda$} as
\begin{equation*}
  \Delta(\lambda) = \lambda_1 + \dots + \lambda_m,
\end{equation*}
and
\begin{equation*}
  p_\lambda = L_{-\lambda_1}\dots L_{-\lambda_m} \in \mathbb{C}[L_{-1}, L_{-2}, \dots].
\end{equation*}

For a composition $\lambda$ with exactly $m$ elements greater than $1$ and exactly $n$ elements equal to $1$, we define $\lambda'$ as the composition obtained from $\lambda$ by removing all $1$s.
We define the \emph{conformal length of $\lambda$} as
\begin{equation*}
  \clen(\lambda) = 2m + n
\end{equation*}
and
\begin{equation*}
  u_\lambda = p_{\lambda'}L_{-1}^n \in \bigoplus_{k \in \mathbb{N}}\mathbb{C}[L_{-2}, L_{-3}, \dots]L_{-1}^k.
\end{equation*}

A \emph{partition (of $n \in \mathbb{N}$)} is a composition $\lambda = [\lambda_1, \dots, \lambda_m]$ (of $n$) such that $\lambda_1 \ge \dots \ge \lambda_m$.
The notation $\lambda \vdash$ means $\lambda$ is a partition, and for $n \in \mathbb{N}$, the notation $\lambda \vdash n$ means $\lambda$ is a partition of $n$.
This notation also tacitly assumes that $\lambda$ is a partition and not just a composition.

Given a composition $\lambda$ of length $m$, there is a unique partition $\prt(\lambda)$ obtained by reordering the entries of $\lambda$ in nonincreasing order.

By the PBW theorem, the set
\begin{equation*}
  \{L_\lambda|c, h\rangle \mid \lambda \vdash\}
\end{equation*}
is a vector space basis of $M(c, h)$.
The representation $M(c, h)$ has a unique maximal proper subrepresentation $J(c, h)$, and the quotient
\begin{equation*}
  L(c, h) = M(c, h)/J(c, h)
\end{equation*}
is the \emph{irreducible highest weight representation} of $\Vir$ with highest weight $(c, h)$.

\section{Modules over algebras with a filtration}
\label{sec:modul-over-algebr}

Let $A$ be an associative (not necessarily commutative) algebra with unit $1$ and filtration $(A^p)_{p \in \mathbb{Z}}$ of subspaces of $A$ such that:
\begin{enumerate}
\item $A^p = 0$ for $p < 0$;
\item $1 \in A^0$;
\item $A^0 \subseteq A^1 \subseteq \dots$;
\item $A = \bigcup_{p \in \mathbb{N}}A^p$;
\item $A^pA^q \subseteq A^{p + q}$ for $p, q \in \mathbb{Z}$.
\end{enumerate}
Let
\begin{equation*}
  \gr(A) = \bigoplus_{p \in \mathbb{N}}A^p/A^{p - 1}
\end{equation*}
be the associated graded vector space.
The vector space $\gr(A)$ is an associative algebra with unit and multiplication given as follows.
For $p, q \in \mathbb{N}$, $a \in A^p$ and $b \in A^q$, we set
\begin{equation*}
  \gamma^p(a)\gamma^q(b) = \gamma^{p + q}(ab),
\end{equation*}
where $\gamma^p: A^p \to \gr(A)$ is the \emph{principal symbol map}, which is the composition of the natural maps $A^p \twoheadrightarrow A^p/A^{p - 1}$ and $A^p/A^{p - 1} \hookrightarrow \gr(A)$.
The unit of $\gr(A)$ is $\gamma^0(1)$.

Let $\partial$ be a derivation of $A$ respecting the filtration $(A^p)_{p \in \mathbb{Z}}$, i.e., it satisfies:
\begin{enumerate}
\item $\partial(ab) = \partial(a)b + a\partial(b)$ for $a, b \in A$;
\item $\partial(A^p) \subseteq A^p$ for $p \in \mathbb{Z}$.
\end{enumerate}
We can define
\begin{align*}
  \partial: \gr(A) &\to \gr(A), \\
  \partial(\gamma^p(a)) &= \gamma^p(\partial(a)) \quad \text{for $p \in \mathbb{Z}$ and $a \in A^p$},
\end{align*}
and it is a derivation of $\gr(A)$.

\begin{example}[\emph{PBW filtration of $U(\mathfrak{g})$}]
  \label{exa:1}
  Let $\mathfrak{g}$ be a Lie algebra.
  The PBW filtration of $U(\mathfrak{g})$, the universal enveloping algebra of $\mathfrak{g}$, is given by
  \begin{equation*}
    U(\mathfrak{g})^p = \vspan\{x_1x_2\dots x_s \mid s \le p, x_1, \dots, x_s \in \mathfrak{g}\} \quad \text{for $p \in \mathbb{Z}$}.
  \end{equation*}
  This filtration clearly satisfies axioms (i)--(v) above, and $\gr(U(\mathfrak{g}))$ is naturally isomorphic to $S(\mathfrak{g})$, the symmetric algebra of $\mathfrak{g}$, which is a polynomial algebra (see \cite[\S2]{dixmier_enveloping_1996} for details).
  Furthermore, if $\partial$ is a derivation of $\mathfrak{g}$ as a Lie algebra, then we can extend $\partial$ uniquely to a derivation of $U(\mathfrak{g})$ (see \cite[Proposition 2.4.9(i)]{dixmier_enveloping_1996}), and it respects the PBW filtration.
  Thus, it defines a derivation $\partial: \gr(U(\mathfrak{g})) \to \gr(U(\mathfrak{g}))$.
\end{example}

Let $M$ be an $A$-module with filtration $(M^p)_{p \in \mathbb{Z}}$ of subspaces of $M$ such that:
\begin{enumerate}
\item $M^p = 0$ for $p < 0$;
\item $M^0 \subseteq M^1 \subseteq \dots$;
\item $M = \bigcup_{p \in \mathbb{N}}M^p$;
\item $A^pM^q \subseteq M^{p + q}$ for $p, q \in \mathbb{Z}$.
\end{enumerate}

Let
\begin{equation*}
  \gr(M) = \bigoplus_{p \in \mathbb{N}}M^p/M^{p - 1}
\end{equation*}
be the associated graded vector space.
Then $\gr(M)$ is a $\gr(A)$-module with operations given as follows.
For $p, q \in \mathbb{N}$, $a \in A^p$ and $u \in M^q$, we set
\begin{equation*}
  \gamma^p(a)\gamma^q_M(u) = \gamma^{p + q}_M(au),
\end{equation*}
where $\gamma^p_M: M^p \to \gr(M)$ is the principal symbol map, which is the composition of the natural maps $M^p \twoheadrightarrow M^p/M^{p - 1}$ and $M^p/M^{p - 1} \hookrightarrow \gr(M)$.

Let $H \in \End(M)$ be a diagonalizable operator.
Defining the eigenspaces of $M$ as
\begin{equation*}
  M_\Delta = \ker(H - \Delta\Id_M) \quad \text{for $\Delta \in \mathbb{C}$},
\end{equation*}
we have
\begin{equation*}
  M = \bigoplus_{\Delta \in \mathbb{C}}M_\Delta.
\end{equation*}
We say the filtration $(M^p)_{p \in \mathbb{Z}}$ is \emph{graded} with respect to the diagonalizable operator $H$ if each subspace $M^p$ is graded, i.e., if
\begin{equation}
  \label{eq:2}
  H(M^p) \subseteq M^p \quad \text{for $p \in \mathbb{Z}$}.
\end{equation}
We assume we have a diagonalizable operator $H$ and a graded filtration $(M^p)_{p \in \mathbb{Z}}$.
By \eqref{eq:2} and \cite[Corollary 1.1]{kac_bombay_2013}, we have
\begin{equation}
  \label{eq:3}
  M^p = \sum_{\Delta \in \mathbb{C}}M^p \cap M_\Delta \quad \text{for $p \in \mathbb{Z}$}.
\end{equation}

By property \eqref{eq:2}, we can define an operator $H \in \End(\gr(M))$ as $H(\gamma^p_M(u)) = \gamma^p_M(Hu)$ for $p \in \mathbb{N}$ and $u \in M^p$.
For $p \in \mathbb{N}$ and $\Delta \in \mathbb{C}$, we define $M^p_\Delta = M^p \cap M_\Delta$.
For $\Delta \in \mathbb{C}$, we define $\gr(M)_\Delta = \bigoplus_{p \in \mathbb{N}}\gamma_M^p(M^p_\Delta)$.
Then $Hu = \Delta u$ for $u \in \gr(M)_\Delta$, and the family of subspaces $(\gr(M)_\Delta)_{\Delta \in \mathbb{C}}$ is linearly independent.
By \eqref{eq:3}, we have the refined grading
\begin{equation}
  \label{eq:4}
  \gr(M) = \bigoplus_{\substack{p \in \mathbb{N} \\ \Delta \in \mathbb{C}}}\gamma^p_M(M^p_\Delta).
\end{equation}

By \eqref{eq:4}, when $\dim(M_\Delta) < \infty$ for $\Delta \in \mathbb{C}$, it is natural to define the \emph{character of $M$} and the \emph{refined character of $M$} as:
\begin{align*}
  \ch_M(q) &= \sum_{\Delta \in \mathbb{C}}\dim(M_\Delta)q^\Delta = \sum_{\Delta \in \mathbb{C}}\dim(\gr(M)_\Delta)q^\Delta, \\
  \ch_M(t, q) &= \sum_{\substack{p \in \mathbb{N} \\ \Delta \in \mathbb{C}}}\dim(\gamma_M^p(M^p_\Delta))t^pq^\Delta.
\end{align*}

The category of modules over $A$ with the given filtration $(A^p)_{p \in \mathbb{Z}}$ is given by modules $M$ with filtration $(M^p)_{p \in \mathbb{Z}}$ satisfying conditions (i)--(iv) above.
A homomorphism $f: M \to N$ must satisfy $f(M^p) \subseteq N^p$ for $p \in \mathbb{Z}$.

If $f: M_1 \to M_2$ is a homomorphism of $A$-modules, then
\begin{align*}
  \gr(f): \gr(M_1) &\to \gr(M_2), \\
  \gr(f)(\gamma_{M_1}^p(u)) &= \gamma_{M_2}^p(f(u)) \quad \text{for $p \in \mathbb{N}$ and $u \in M_1^p$}
\end{align*}
defines a homomorphism of $\gr(A)$-modules.
If the filtrations of $M_1$ and $M_2$ are graded, then we require that $f$ respects the gradings of $M_1$ and $M_2$, i.e., we require that $f\circ H^{M_1} = H^{M_2}\circ f$, and this implies that $\gr(f)$ also respects the gradings of $\gr(M_1)$ and $\gr(M_2)$.

Therefore, we obtain a functor
\begin{equation*}
  \gr: \{\text{(graded) $A$-modules}\} \to \{\text{(graded) $\gr(A)$-modules}\}.
\end{equation*}

From now on, some subscripts or superscripts will be omitted, so for example $\gamma^p_M$ simplifies to $\gamma^p$.

The following simple linear algebra proposition relates a basis of the vector space $M$ to a basis of its associated graded $\gr(M)$, and the proof is straightforward.

\begin{proposition}
  \label{prp:1}
  Let $M$ be a vector space, let $(M^p)_{p \in \mathbb{Z}}$ be a family of subspaces of $M$ satisfying properties \emph{(i)--(iii)} above, and let $(B^p)_{p \in \mathbb{N}}$ be an increasing family of sets satisfying $B^p \subseteq M^p$ for $p \in \mathbb{N}$.
  Then $\{v + M^{p - 1} \mid v \in B^p \setminus B^{p - 1}\}$ spans (resp.\ is a basis of) $M^p/M^{p - 1}$ for $p \in \mathbb{N}$ if and only if $B^p$ spans (resp.\ is a basis of) $M^p$ for $p \in \mathbb{N}$ ($B^{-1} = \emptyset$).
  In particular, if $\{\gamma^p(v) \mid p \in \mathbb{N}, v \in B^p \setminus B^{p - 1}\}$ is a basis of $\gr(M)$, then $\bigcup_{p \in \mathbb{N}}B^p$ is a basis of $M$.
\end{proposition}

\begin{example}[$\gr(M(c, h))$]
  \label{exa:2}
  We consider the subalgebra $\Vir^- = \bigoplus_{n < 0}\mathbb{C}L_n$ of the Virasoro Lie algebra $\Vir$.
  We have a natural isomorphism $\mathbb{C}[L_{-1}, L_{-2}, \dots] \cong \gr(U(\Vir^-))$, and we consider the Verma module $M(c, h)$ as a module over $U(\Vir^-)$ acting by left multiplication.
  The filtration of $M(c, h)$ is given by $M(c, h)^p = U(\Vir^-)^p|c, h\rangle$ for $p \in \mathbb{Z}$.
  Let $(c, h) \in \mathbb{C}^2$.
  By \cite[\S2]{dixmier_enveloping_1996}, we have a natural isomorphism
  \begin{align*}
    \gr(M(c, h)) &\xrightarrow{\sim} \mathbb{C}[L_{-1}, L_{-2}, \dots] \cong \gr(U(\Vir^-)), \\
    \gamma^{\len(\lambda)}(L_\lambda|c, h\rangle) & \mapsto p_\lambda \quad \text{for $\lambda$ a composition}.
  \end{align*}
  The diagonalizable operator $H = L_0$ acts by left multiplication.
  For the Verma module $M(c, h)$, the characters are given by:
  \begin{align*}
    \ch_{M(c, h)}(t, q) &= \frac{q^h}{\prod_{k \in \mathbb{Z}_+}(1 - tq^k)}, \\
    \ch_{M(c, h)}(q) &= \frac{q^h}{\prod_{k \in \mathbb{Z}_+}(1 - q^k)} = \frac{q^h}{(q)_\infty}.
  \end{align*}
\end{example}

\begin{example}[$\ch_{L(1/2, 0)}(t, q)$]
  \label{exa:3}
  In \cite{andrews_singular_2022}, the refined character of the \emph{Virasoro minimal model} $\Vir_{3, 4} = L(1/2, 0)$, also known as the \emph{Ising model}, is computed explicitly.
  Interestingly, this character is connected to Nahm sums for the matrix $\left(\begin{smallmatrix} 8 & 3 \\ 3 & 2 \end{smallmatrix}\right)$ (see \cite{Nahm2007}) and can be explicitly expressed as
  \begin{equation*}
    \ch_{L(1/2, 0)}(t, q) = \sum_{k_1, k_2 \in \mathbb{N}}t^{2k_1 + k_2}\frac{q^{4k_1^2 + 3k_1k_2 + k_2^2}}{(q)_{k_1}(q)_{k_2}}(1 - q^{k_1} + q^{k_1 + k_2}),
  \end{equation*}
  where $(q)_k = \prod_{j = 1}^k(1 - q^j) \in \mathbb{C}[q]$ denotes the \emph{$q$-Pochhammer symbol}.
\end{example}

We pick a highest weight $(c, h)$.
We have a natural epimorphism of $U(\Vir^-)$-modules
\begin{align*}
  \pi_{c, h}: M(c, h) &\twoheadrightarrow L(c, h), \\
  \pi_{c, h}(u) &= u + J(c, h),
\end{align*}
and it satisfies $\ker(\pi_{c, h}) = J(c, h)$.
Applying the functor $\gr$, we obtain an epimorphism of modules over $\gr(U(\Vir^-))$
\begin{equation*}
  \gr(\pi_{c, h}): \gr(M(c, h)) \twoheadrightarrow \gr(L(c, h)),
\end{equation*}
and this produces a natural isomorphism of modules over $\gr(U(\Vir^-))$
\begin{equation*}
  \gr(M(c, h))/K(c, h) \xrightarrow{\sim} \gr(L(c, h)),
\end{equation*}
where
\begin{equation*}
  K(c, h) = \ker(\gr(\pi_{c, h})).
\end{equation*}
Explicitly, we have
\begin{equation}
  \label{eq:5}
  K(c, h) = \sum_{p \in \mathbb{N}}\gamma^p(J(c, h) \cap M(c, h)^p).
\end{equation}

\section{Fields over vector spaces}
\label{sec:fields-over-vector}

The vector space of \emph{formal distributions in $n \in \mathbb{N}$ variables}, denoted by $\mathbb{C}[[x_1^{\pm1}, \dots, x_n^{\pm1}]]$, is the set of functions $f: \mathbb{Z}^n \to \mathbb{C}$, written as $f(x_1, \dots, x_n) = \sum_{m_1, \dots, m_n \in \mathbb{Z}}f_{m_1, \dots, m_n}x_1^{m_1}\dots x_n^{m_n}$, with the natural operations of addition and multiplication by a scalar.
The field of \emph{formal Laurent series}, denoted by $\mathbb{C}((x))$, is the subspace of elements $f(x) \in \mathbb{C}[[x^{\pm1}]]$ such that there is $N \in \mathbb{Z}$ with $f_n = 0$ for $n \le N$.
We also have $\mathbb{C}((x)) = \Frac(\mathbb{C}[[x]])$, so $\mathbb{C}((x))$ is actually a field.
If $V$ is a vector space, we similarly define $V[[x_1^{\pm1}, \dots, x_n^{\pm1}]]$ and $V((x))$, but in this case, $V((x))$ is only a vector space.

Let $V$ be a vector space.
The \emph{Fourier expansion of a formal distribution $a(z) \in V[[z^{\pm1}]]$}, written as $a(z) = \sum_{n \in \mathbb{Z}}a_nz^n$, is conventionally written in the theory of vertex algebras as
\begin{equation*}
  a(z) = \sum_{n \in \mathbb{Z}}a_{(n)}z^{-n - 1},
\end{equation*}
where
\begin{equation*}
  a_{(n)} = a_{-n - 1}.
\end{equation*}

Let $V$ be a vector space, and let $a(z) \in \End(V)[[z^{\pm1}]]$ be a formal distribution.
We set:
\begin{align*}
  a(z)_+ &= \sum_{n \le -1}a_{(n)}z^{-n - 1}, \\
  a(z)_- &= \sum_{n \ge 0}a_{(n)}z^{-n - 1}.
\end{align*}

A formal distribution $a(z)$ is a \emph{field} if
\begin{equation*}
  a(z)b = \sum_{n \in \mathbb{Z}}a_{(n)}bz^{-n - 1} \in V((z)) \quad \text{for $b \in V$}.
\end{equation*}
The vector space of fields over $V$ is denoted by $\mathcal{F}(V)$.
We note that
\begin{equation*}
  \mathcal{F}(V) = \Hom(V, V((z))).
\end{equation*}
Therefore, we can define a field $a(z)$ by defining $a(z)b \in V((z))$ for $b \in V$.

\begin{proposition}[{\cite[Proposition 3.3.2]{nozaradan_introduction_2008}}]
  \label{prp:2}
  Let $a(z), b(z) \in \mathcal{F}(V)$ be two fields.
  Then $:a(z)b(z): \in \End(V)[[z^{\pm1}]]$ is again a field, where $:a(z)b(z):$ is defined by
  \begin{equation*}
    :a(z)b(z):c = a(z)_+b(z)c + b(z)a(z)_-c \quad \text{for $c \in V$}.
  \end{equation*}
\end{proposition}

We have thus defined the notion of \emph{normal ordered product between fields $a(z), b(z) \in \mathcal{F}(V)$}, denoted by $:a(z)b(z):$.
In general, the operation of normal ordered product is neither commutative nor associative.
We follow the convention that the normal ordered product is read from right to left, so that, by definition,
\begin{equation*}
  :a(z)b(z)c(z): = :a(z)(:b(z)c(z):):.
\end{equation*}

\begin{lemma}[{\cite[Proposition 3.3.3]{nozaradan_introduction_2008}}]
  \label{lmm:1}
  Let $a(z), b(z) \in \mathcal{F}(V)$ be two fields.
  Their normal ordered product is written explicitly as
  \begin{equation*}
    :a(z)b(z): = \sum_{j \in \mathbb{Z}}:a(z)b(z):_{(j)}z^{-j - 1},
  \end{equation*}
  with
  \begin{equation*}
    :a(z)b(z):_{(j)}c = \sum_{n \le -1}a_{(n)}b_{(j - n - 1)}c + \sum_{n \ge 0}b_{(j - n - 1)}a_{(n)}c \quad \text{for $c \in V$}.
  \end{equation*}
\end{lemma}

\begin{lemma}
  \label{lmm:2}
  Let $V$ be a vector space.
  We consider $s$ fields $a^1(z), \dots, a^s(z) \in \mathcal{F}(V)$, with $s \ge 2$, and let $b \in V$.
  For $l \in \mathbb{Z}$,
  \begin{equation*}
    :a^1(z)a^2(z)\dots a^s(z):_{(l)}b = \sum_{n_1, \dots, n_{s - 1} \in \mathbb{N}}\sum_{k = 0}^{s - 1}R^{l, k}_{n_1, \dots, n_{s - 1}}(a^1(z), \dots, a^s(z))b,
  \end{equation*}
  where $R^{l, k}_{n_1, \dots, n_{s - 1}}(a^1(z), \dots, a^s(z))$ is the sum of $\binom{s - 1}{k}$ terms given by
  \begin{align*}
    &R^{l, k}_{n_1, \dots, n_{s - 1}}(a^1(z), \dots, a^s(z)) = \\
    &\sum_{\substack{1 \le i_1 < \dots < i_k \le s - 1 \\ 1 \le j_1 < \dots < j_{s - 1 - k} \le s - 1 \\ \{i_1, \dots, i_k\} \cup \{j_1, \dots, j_{s - 1 - k}\} = \{1, \dots, s - 1\}}}a^{j_1}_{(-n_{j_1} - 1)}\dots a^{j_{s - 1 - k}}_{(-n_{j_{s - 1 - k}} - 1)}a^s_{(l - k - \sum_{r = 1}^kn_{i_r} + \sum_{r = 1}^{s - 1 - k}n_{j_r})}a^{i_k}_{(n_{i_k})}\dots a^{i_1}_{(n_{i_1})}.
  \end{align*}
\end{lemma}

\begin{proof}
  This follows from \zcref{lmm:1} and induction on $s$.
\end{proof}

\section{Modules over the simple Virasoro vertex algebras}
\label{sec:modules-over-simple}

Let $c \in \mathbb{C}$.
We make the subalgebra $\Vir^\ge \oplus \mathbb{C}L_{-1}$ of $\Vir$ act on $\mathbb{C}$ as follows:
\begin{equation*}
  \text{$L_n1 = 0$ for $n \ge -1$ and $C1 = c$}.
\end{equation*}
The induced $\Vir$-module
\begin{equation*}
  \Vir^c = \Ind^\Vir_{\Vir^\ge \oplus \mathbb{C}L_{-1}}(\mathbb{C}) = U(\Vir) \otimes_{U(\Vir^\ge \oplus \mathbb{C}L_{-1})} \mathbb{C}
\end{equation*}
carries the structure of a vertex algebra and is called the \emph{universal Virasoro vertex algebra of central charge $c$}, see \cite[\S2]{frenkel_vertex_2001} or \cite[\S2]{callegaro_introduction_2017-1} for details.

The vertex algebra $\Vir^c$ has a unique maximal proper ideal.
We denote by $\Vir_c$ the unique simple quotient, and we call it the \emph{simple Virasoro vertex algebra of central charge $c$}.
Let $p, q \ge 2$ be relatively prime integers, and we set
\begin{equation*}
  c_{p, q} = 1 - \frac{6(p - q)^2}{pq}.
\end{equation*}

Let $V$ be a $\Vir$-module.
A vector $u$ in $V$ is called \emph{singular} if it is nonzero and
\begin{equation*}
  L_nu = 0 \quad \text{for $n \in \mathbb{Z}_+$}.
\end{equation*}

\begin{theorem}[{\cite{feigin_verma_1984} and \cite{gorelik_simplicity_2007}}]
  \label{thr:4}
  The following are equivalent:
  \begin{enumerate}
  \item $\Vir^c$ is not simple, i.e., $\Vir^c \neq \Vir_c$;
  \item $c$ is of the form $c_{p, q}$ for some $p, q \ge 2$ relatively prime integers.
  \end{enumerate}
  Moreover, the maximal proper ideal of $\Vir^{c_{p, q}}$ is generated by a singular vector of conformal weight $(p - 1)(q - 1)$, denoted by $a_{p, q}$.
  In the expression
  \begin{equation*}
    a_{p, q} = \sum_{\substack{\lambda \vdash (p - 1)(q - 1) \\ \lambda_{\len(\lambda)} \ge 2}}c_\lambda L_\lambda\vac,
  \end{equation*}
  where $c_\lambda \in \mathbb{Q}$, the coefficient of $L_{-2}^{(p - 1)(q - 1)/2}$ is nonzero.
\end{theorem}

\begin{proposition}
  \label{prp:3}
  Let $V$ be a vertex algebra, and let $Y^M: V \to \mathcal{F}(M)$ be a $V$-module.
  For $s, n_1, \dots, n_s \in \mathbb{N}$ and $a^1, \dots, a^s \in V$,
  \begin{equation*}
    Y^M(a^1_{(-n_1 - 1)}\dots a^s_{(-n_s - 1)}\vac, z) = \frac{:\partial^{n_1}_zY^M(a^1, z)\dots\partial^{n_s}_zY^M(a^s, z):}{n_1!\dots n_s!}.
  \end{equation*}
\end{proposition}

\begin{proof}
  This is a consequence of the $n$-product identity for modules over vertex algebras (\cite[(5.2.16)]{lepowsky_introduction_2004}).
\end{proof}

We are interested in the irreducible modules over the simple Virasoro vertex algebras $\Vir_{p, q} = \Vir_{c_{p, q}} = L(c_{p, q}, 0)$, where $p, q \ge 2$ are relatively prime integers.
In this context, the following constants play a crucial role.
For integers $m, n$ such that $0 < m < p$ and $0 < n < q$, we set
\begin{equation*}
  h_{m, n} = \frac{(np - mq)^2 - (p - q)^2}{4pq}.
\end{equation*}

\begin{theorem}[{\cite{wang_rationality_1993}}]
  \label{thr:5}
  We set $c = c_{p, q}$ for some $p, q \ge 2$ relatively prime integers.
  Then the irreducible modules over $\Vir_c$ are $L(c, h_{m, n})$ for integers $m, n$ such that $0 < m < p$ and $0 < n < q$.
  Let $Y^{L(c, h_{m, n})}_{\Vir^c}: \Vir^c \to \mathcal{F}(L(c, h_{m, n}))$ be the state-field correspondence of $L(c, h_{m, n})$ as a module over $\Vir^c$, and let $Y^{L(c, h_{m, n})}_{\Vir_c}: \Vir_c \to \mathcal{F}(L(c, h_{m, n}))$ be the state-field correspondence of $L(c, h_{m, n})$ as a module over $\Vir_c$.
  Then $Y^{L(c, h_{m, n})}_{\Vir^c}$ factors through $Y^{L(c, h_{m, n})}_{\Vir_c}$, i.e., the following diagram commutes
  \begin{equation*}
    \begin{tikzcd}
      \Vir^c \arrow[r, two heads] \arrow[rd, "{Y^{L(c, h_{m, n})}_{\Vir^c}}"'] & {\Vir_c} \arrow[d, "{Y^{L(c, h_{m, n})}_{\Vir_c}}"] \\
      & {\mathcal{F}(L(c, h_{m, n}))}
    \end{tikzcd}
  \end{equation*}
  where the horizontal arrow is the quotient map.
\end{theorem}

In this article, we consider:
\begin{align*}
  p &= 2, \\
  q &= 2s + 1 \quad \text{for $s \in \mathbb{Z}_+$}.
\end{align*}
By \zcref{thr:5}, $\Vir_{2, 2s + 1}$ has $s$ irreducible modules $L(c_{2, 2s + 1}, h_{1, 1}), \dots, L(c_{2, 2s + 1}, h_{1, s})$, known collectively as the \emph{boundary minimal models}.

\section{PBW bases and characters of the boundary minimal models}
\label{sec:pbw-bases-characters}

In this section, we fix $s \in \mathbb{Z}_+$ and $i \in \{1, \dots, s\}$.

A partition $\lambda = [\lambda_1, \dots, \lambda_m]$ \emph{contains a partition} $\eta = [\eta_1, \dots, \eta_n]$, written as $\eta \subseteq \lambda$, if $m \ge n$ and there is $j \in \mathbb{Z}_+$ such that $1 \le j \le m - n + 1$ and $[\lambda_j, \lambda_{j + 1}, \dots, \lambda_{j + n - 1}] = \eta$.

We define
\begin{equation*}
  p^{s, i}(t, q) = \sum_{\lambda \in P^{s, i}}t^{\len(\lambda)}q^{\Delta(\lambda)} \in \mathbb{N}[[t, q]],
\end{equation*}
where $P^{s, i}$ is the set of partitions that do not contain any partition in $R^{s, i}$ as defined in \zcref{sec:introduction}, i.e.,
\begin{equation*}
  P^{s, i} = \{\lambda \vdash \mid \text{for $\eta \in R^{s, i}$, $\lambda \nsupseteq \eta$}\}.
\end{equation*}

\begin{lemma}[{\cite[\S7]{andrews_theory_1998}}]
  \label{lmm:3}
  We have
  \begin{equation*}
    p^{s, i}(t, q) = \sum_{k = (k_1, \dots, k_{s - 1}) \in \mathbb{N}^{s - 1}}t^{kB^{(s)}_{s - 1}}\frac{q^{\frac{1}{2}kG^{(s)}k^\top + kB^{(s)}_{s - i}}}{(q)_{k_1}\dots(q)_{k_{s - 1}}}.
  \end{equation*}
\end{lemma}

\begin{lemma}[{\cite{feigin_verma_1984} and \cite[\S7]{andrews_theory_1998}}]
  \label{lmm:4}
  We have
  \begin{equation*}
    \ch_{L(c_{2, 2s + 1}, h_{1, i})}(q) = q^{h_{1, i}}p^{s, i}(1, q) = q^{h_{1, i}}\left(\prod^{\infty}_{\substack{n = 1 \\ n \not\equiv 0, \pm i \bmod 2s + 1}}(1 - q^n)^{-1}\right).
  \end{equation*}
\end{lemma}

We order the PBW basis of $U(\Vir^-) = \vspan\{L_{\lambda} \mid \lambda \vdash\}$ by degree reverse lexicographic order with $L_{-1} > L_{-2} > \dots$.
Formally, for any partitions $\lambda$ and $\eta$, we define
\begin{equation*}
  L_\lambda \le L_\eta\text{ if and only if }p_\lambda \le p_\eta.
\end{equation*}
For $x \in U(\Vir^-)$ with $x \neq 0$, we may write
\begin{equation*}
  x = c_1L_{\lambda^1} + c_2L_{\lambda^2} + \dots + c_rL_{\lambda^r},
\end{equation*}
where for $1 \le j \le r$, $0 \neq c_j \in \mathbb{C}$ and $L_{\lambda^1} > L_{\lambda^2} > \dots > L_{\lambda^r}$.
We define the \emph{leading power of $x$} as $\lp(x) = L_{\lambda^1}$.
We set $\lp(0) = 0$.
Next, for a highest weight $(c, h)$, we extend the definition of $\lp$ from $U(\Vir^-)$ to $M(c, h)$ by considering the isomorphism of vector spaces $U(\Vir^-) \xrightarrow{\sim} M(c, h), L_\lambda \mapsto L_\lambda|c, h\rangle$, where $\lambda$ is a partition.

\begin{remark}
  \label{rmk:1}
  The definition of the order in the PBW basis of $U(\Vir^-)$ was made so that for a highest weight $(c, h)$, a partition $\lambda$ and $u \in M(c, h)$, if $\lp(u) = L_\lambda|c, h\rangle$, then $\lp(\gamma^{\len(\lambda)}(u)) = p_\lambda$.
\end{remark}

\begin{lemma}
  \label{lmm:5}
  For $\lambda \in R^{s, i}$, there exists $u \in J(c_{2, 2s + 1}, h_{1, i})$ homogeneous such that $\lp(u) = L_\lambda|c_{2, 2s + 1}, h_{1, i}\rangle$.
\end{lemma}

\begin{proof}
  By \zcref{thr:4}, the maximal proper ideal of $\Vir^{c_{2, 2s + 1}}$ is generated by a singular vector of conformal weight $2s$, denoted by $a_{2, 2s + 1}$, which has the form
  \begin{equation}
    \label{eq:6}
    a_{2, 2s + 1} = L_{-2}^s\vac + \sum_{\substack{\mu \vdash 2s \\ \mu_{\len(\mu)} \ge 2 \\ \len(\mu) < s}}c_\mu L_\mu\vac,
  \end{equation}
  where $c_\mu \in \mathbb{Q}$.
  We have two cases:
  \begin{enumerate}
  \item $\lambda = [a^d, (a + 1)^{s - d}]$ for some $a \in \mathbb{Z}_+$ and $d \in \{1, \dots, s\}$.
    We set
    \begin{equation*}
      u = (a_{2, 2s + 1})_{(s - sa + d - 1)}|c_{2, 2s + 1}, h_{1, i}\rangle
    \end{equation*}
    and claim
    \begin{equation*}
      \lp(u) = L_\lambda|c_{2, 2s + 1}, h_{1, i}\rangle.
    \end{equation*}
    To prove this, we first prove that
    \begin{equation}
      \label{eq:7}
      (L_{-2}^s\vac)_{(s - sa + d - 1)}|c_{2, 2s + 1}, h_{1, i}\rangle = \binom{s}{d}L_\lambda|c_{2, 2s + 1}, h_{1, i}\rangle + \sum_{\substack{\mu \vdash \Delta(\lambda) \\ \len(\mu) \le s \\ \mu \neq \lambda}}c_\mu L_\mu|c_{2, 2s + 1}, h_{1, i}\rangle,
    \end{equation}
    where $c_\mu \in \mathbb{Q}$.
    Applying \zcref{lmm:2} and \zcref{prp:3} with:
    \begin{align*}
      a^1(z) = \dots = a^s(z) &= L(z) = \sum_{n \in \mathbb{Z}}L_{(n)}z^{-n - 1} = \sum_{n \in \mathbb{Z}}L_nz^{-n - 2}, \\
      b &= |c_{2, 2s + 1}, h_{1, i}\rangle, \\
      l &= s - sa + d - 1,
    \end{align*}
    we obtain
    \begin{equation}
      \label{eq:8}
      (L_{-2}^s\vac)_{(l)}|c_{2, 2s + 1}, h_{1, i}\rangle = \sum_{n_1, \dots, n_{s - 1} \in \mathbb{N}}\sum_{k = 0}^{s - 1}R^{l, k}_{n_1, \dots, n_{s - 1}}(L(z), \dots, L(z))|c_{2, 2s + 1}, h_{1, i}\rangle,
    \end{equation}
    where $R^{l, k}_{n_1, \dots, n_{s - 1}}(L(z), \dots, L(z))$ is the sum of $\binom{s - 1}{k}$ terms given by
    \begin{align*}
      &R^{l, k}_{n_1, \dots, n_{s - 1}}(L(z), \dots, L(z)) = \\
      &\sum_{\substack{1 \le i_1 < \dots < i_k \le s - 1 \\ 1 \le j_1 < \dots < j_{s - 1 - k} \le s - 1 \\ \{i_1, \dots, i_k\} \cup \{j_1, \dots, j_{s - 1 - k}\} = \{1, \dots, s - 1\}}}L_{(-n_{j_1} - 1)}\dots L_{(-n_{j_{s - 1 - k}} - 1)}L_{(l - k - \sum_{r = 1}^kn_{i_r} + \sum_{r = 1}^{s - 1 - k}n_{j_r})}L_{(n_{i_k})}\dots L_{(n_{i_1})},
    \end{align*}
    where
    \begin{align*}
      &L_{(-n_{j_1} - 1)}\dots L_{(-n_{j_{s - 1 - k}} - 1)}L_{(l - k - \sum_{r = 1}^kn_{i_r} + \sum_{r = 1}^{s - 1 - k}n_{j_r})}L_{(n_{i_k})}\dots L_{(n_{i_1})} = \\
      &L_{-n_{j_1} - 2}\dots L_{-n_{j_{s - 1 - k}} - 2}L_{s - sa + d - 2 - k - \sum_{r = 1}^kn_{i_r} + \sum_{r = 1}^{s - 1 - k}n_{j_r}}L_{n_{i_k} - 1}\dots L_{n_{i_1} - 1}.
    \end{align*}
    We see that each term $L_\mu$ appearing in $R^{l, k}_{n_1, \dots, n_{s - 1}}(L(z), \dots, L(z))$, where $\mu$ is a composition (not necessarily a partition), satisfies $\Delta(\mu) = sa + s - d = \Delta(\lambda)$ and $\len(\mu) = s = \len(\lambda)$.

    We recall that $L_\eta - L_{\eta\sigma} \in U(\Vir^-)^{\len(\eta) - 1}$ for a composition $\eta$ and a permutation $\sigma \in S_{\len(\eta)}$ (see \cite[Lemma 2.1.5]{dixmier_enveloping_1996} for details).
    Therefore, for a composition $\eta$, we can expand $L_\eta$ in the following way
    \begin{equation*}
      L_\eta = L_{\prt(\eta)} + \sum_{\substack{\mu \vdash \Delta(\eta) \\ \len(\mu) < \len(\eta)}}c_\mu L_\mu,
    \end{equation*}
    where $c_\mu \in \mathbb{Q}$.

    We need to compute the coefficient of $L_\lambda|c_{2, 2s + 1}, h_{1, i}\rangle$ in $u$ when expressed as a linear combination of the elements of the usual PBW basis.
    To do this, we need to see how much each sum $R^{l, k}_{n_1, \dots, n_{s - 1}}(L(z), \dots, L(z))|c_{2, 2s + 1}, h_{1, i}\rangle$ contributes to the coefficient of $L_\lambda|c_{2, 2s + 1}, h_{1, i}\rangle$ in $u$.
    We have two subcases:
    \begin{description}[leftmargin=!]
    \item[Subcase $a = 1$] In the sum \eqref{eq:8}, if $k > d$ or $k < d - 1$ or $n_j \neq 0$ for some $j = 1, \dots, s - 1$, then $R^{l, k}_{n_1, \dots, n_{s - 1}}(L(z), \dots, L(z))|c_{2, 2s + 1}, h_{1, i}\rangle$ contributes $0$.
      As for the two remaining sums, the sum $R^{l, d - 1}_{0, \dots, 0}(L(z), \dots, L(z))|c_{2, 2s + 1}, h_{1, i}\rangle$ contributes $\binom{s - 1}{d - 1}$ because we are picking $d - 1$ elements from $\{1, \dots, s - 1\}$, and the sum $R^{l, d}_{0, \dots, 0}(L(z), \dots, L(z))|c_{2, 2s + 1}, h_{1, i}\rangle$ contributes $\binom{s - 1}{d}$ because we are picking $d$ elements from $\{1, \dots, s - 1\}$.
      Since $\binom{s - 1}{d - 1} + \binom{s - 1}{d} = \binom{s}{d}$, we obtain \eqref{eq:7}.
    \item[Subcase $a > 1$] In the sum \eqref{eq:8}, if $k > 0$, then $R^{l, k}_{n_1, \dots, n_{s - 1}}(L(z), \dots, L(z))|c_{2, 2s + 1}, h_{1, i}\rangle$ contributes $0$.
      We have
      \begin{equation*}
        R^{l, 0}_{n_1, \dots, n_{s - 1}}(L(z), \dots, L(z)) = L_{-n_1 - 2}\dots L_{-n_{s - 1} - 2}L_{s - sa + d - 2 + \sum_{r = 1}^{s - 1}n_r}.
      \end{equation*}
      Therefore, the sum $R^{l, 0}_{n_1, \dots, n_{s - 1}}(L(z), \dots, L(z))|c_{2, 2s + 1}, h_{1, i}\rangle$ contributes $\binom{s - 1}{d - 1}$ when exactly $d - 1$ of the $n_j$'s are equal to $a - 2$ and $s - d$ of the $n_j$'s are equal to $a - 1$, contributes $\binom{s - 1}{d}$ when exactly $d$ of the $n_j$'s are equal to $a - 2$ and $s - d - 1$ of the $n_j$'s are equal to $a - 1$ and contributes $0$ otherwise.
      Since $\binom{s - 1}{d - 1} + \binom{s - 1}{d} = \binom{s}{d}$, we obtain \eqref{eq:7}.
    \end{description}
    We note that if $\mu$ is a partition of $\Delta(\lambda)$ such that $\len(\mu) \le s$ and $\mu \neq \lambda$, then $L_\lambda > L_\mu$.
    Thus, \eqref{eq:7} implies $\lp((L_{-2}^s\vac)_{(s - sa + d - 1)}|c_{2, 2s + 1}, h_{1, i}\rangle) = L_\lambda|c_{2, 2s + 1}, h_{1, i}\rangle$.
    From \eqref{eq:6} and \zcref{lmm:2}, we see that $\lp(u) = \lp((L_{-2}^s\vac)_{(s - sa + d - 1)}|c_{2, 2s + 1}, h_{1, i}\rangle) = L_\lambda|c_{2, 2s + 1}, h_{1, i}\rangle$.
    Finally, by \zcref{thr:5}, $u \in J(c_{2, 2s + 1}, h_{1, i})$, and it is clear that $u$ is homogeneous from its definition.
  \item $\lambda = [1^i]$.
    This follows from \cite{benoit_degenerate_1988}, where it is proved that one of the generating singular vectors in $J(c_{2, 2s + 1}, h_{1, i})$ contains the term $L_{-1}^i|c_{2, 2s + 1}, h_{1, i}\rangle$ in its expansion with respect to the usual PBW basis. \qedhere
  \end{enumerate}
\end{proof}

\begin{remark}
  \label{rmk:2}
  For any partitions $\lambda$ and $\eta$, if $\lambda \supseteq \eta$, then $p_\eta \mid p_\lambda$.
  The converse is not true.
  For example, $p_{[4, 2]} \mid p_{[4, 3, 2]}$, but $[4, 3, 2] \nsupseteq [4, 2]$.
  However, if $\eta = [\eta_1, \dots, \eta_m]$ and $\eta_1 - \eta_m \le 1$, then $\lambda \supseteq \eta$ if and only if $p_\eta \mid p_\lambda$.
\end{remark}

\begin{lemma}
  \label{lmm:6}
  There is an alternative description for $P^{s, i}$, namely
  \begin{equation*}
    P^{s, i} = \{\lambda \vdash \mid \text{for $\eta \in R^{s, i}$, $p_\eta \nmid p_\lambda$}\}.
  \end{equation*}
\end{lemma}

\begin{proof}
  This is a consequence of \zcref{rmk:2}.
\end{proof}

\begin{proof}[Sketch of proof of \zcref{thr:2}]
  As explained in (the proof of) \cite[Theorem 0.4]{salazar_pbw_2024}, we can only use Gröbner basis theory on the free module $\bigoplus_{n = 1}^N\mathbb{C}[L_{-2}, L_{-3}, \dots, L_{-N}]L_{-1}^n$ for $N \in \mathbb{N}$, and then we let $N \to \infty$ to conclude.
  However, here we limit ourselves to a brief informal sketch to see how the results obtained so far lead to a proof of \zcref{thr:2}.

  Let $G$ be a Gröbner basis of $K(c_{2, 2s + 1}, h_{1, i})$, and we set:
  \begin{align*}
    B &= \{p_\lambda \mid \text{for $a \in G$, $\lp(a) \nmid p_\lambda$}\}, \\
    \overline{B} &= \{p_\lambda \mid \lambda \in P^{s, i}\}.
  \end{align*}
  By \eqref{eq:5}, \zcref{rmk:1}, and \zcref{lmm:5}, we have $\overline{B} \subseteq B$.
  Actually, \zcref{prp:1}, \zcref{lmm:6} and \zcref{lmm:4} imply $\overline{B} = B$.
  We have found a basis of $\gr(L(c_{2, 2s + 1}, h_{1, i}))$ and by \zcref{prp:1}, a basis of $L(c_{2, 2s + 1}, h_{1, i})$.
\end{proof}

\begin{proof}[Proof of \zcref{thr:1}]
  This is a corollary of (the proof of) \zcref{thr:2}.
\end{proof}

\section{The Li filtration}
\label{sec:li-filtr}

Let $V$ be a vertex algebra, and let $(a^i)_{i \in I}$ be a family of strong generators of $V$.
For $p \in \mathbb{Z}$, we set
\begin{equation*}
  F_pV = \vspan\{a^{i_1}_{(-n_1 - 1)}\dots a^{i_s}_{(-n_s - 1)}\vac \mid s, n_1, \dots, n_s \in \mathbb{N}, i_1, \dots, i_s \in I, n_1 + \dots + n_s \ge p\}.
\end{equation*}

\begin{proposition}[{\cite{li_abelianizing_2005}}]
  \label{prp:4}
  The filtration $(F_pV)_{p \in \mathbb{Z}}$ satisfies:
  \begin{enumerate}
  \item $F_pV = V$ for $p \le 0$;
  \item $\vac \in F_0V \supseteq F_1V \supseteq \dots$;
  \item $T(F_pV) \subseteq F_{p + 1}V$ for $p \in \mathbb{Z}$;
  \item $a_{(n)}F_qV \subseteq F_{p + q - n - 1}V$ for $p, q \in \mathbb{Z}$, $a \in F_pV$ and $n \in \mathbb{Z}$;
  \item $a_{(n)}F_qV \subseteq F_{p + q - n}V$ for $p, q \in \mathbb{Z}$, $a \in F_pV$ and $n \in \mathbb{N}$.
  \end{enumerate}
\end{proposition}

Let
\begin{equation*}
  \gr_F(V) = \bigoplus_{p \in \mathbb{N}}F_pV/F_{p + 1}V
\end{equation*}
be the associated graded vector space.
By \cite{li_abelianizing_2005}, the vector space $\gr_F(V)$ is a vertex Poisson algebra with operations given as follows.
For $p, q \in \mathbb{N}$, $a \in F_pV$ and $b \in F_qV$, we set:
\begin{align*}
  \sigma_p(a)\sigma_q(b) &= \sigma_{p + q}(a_{(-1)}b), \\
  T(\sigma_p(a)) &= \sigma_{p + 1}(Ta), \\
  Y_-(\sigma_p(a), z)\sigma_q(b) &= \sum_{n \in \mathbb{N}}\sigma_{p + q - n}(a_{(n)}b)z^{-n - 1},
\end{align*}
where $\sigma_p: F_pV \to \gr_F(V)$ is the \emph{principal symbol map}, which is the composition of the natural maps $F_pV \twoheadrightarrow F_pV/F_{p + 1}V$ and $F_pV/F_{p + 1}V \hookrightarrow \gr_F(V)$.
The unit is $\sigma_0(\vac)$.
The filtration $(F_pV)_{p \in \mathbb{Z}}$ is called the \emph{Li filtration of $V$}.

\begin{example}[$\gr_F(\Vir^c)$]
  \label{exa:4}
  We pick $c \in \mathbb{C}$.
  We have an isomorphism
  \begin{align*}
    \gr_F(\Vir^c) &\xrightarrow{\sim} \mathbb{C}[L_{-2}, L_{-3}, \dots], \\
    \sigma_{\Delta(\lambda) - \clen(\lambda)}(L_\lambda\vac) &\mapsto p_\lambda \quad \text{for $\lambda$ a composition without ones}.
  \end{align*}
  The derivation $T \in \Der(\mathbb{C}[L_{-2}, L_{-3}, \dots])$ is given by $T(L_{-n}) = (n - 1)L_{-n - 1}$ for $n \ge 2$, which is extended to a derivation.
  The Poisson structure of $\gr_F(\Vir^c)$ is trivial (i.e., the map $Y_-$ is zero).
\end{example}

Let $V$ be a vertex algebra, let $(a^i)_{i \in I}$ be a family of strong generators of $V$, and let $M$ be a $V$-module.
For $p \in \mathbb{Z}$, we set
\begin{equation*}
  F_pM = \vspan\{a^{i_1M}_{(-n_1 - 1)}\dots a^{i_sM}_{(-n_s - 1)}u \mid s, n_1, \dots, n_s \in \mathbb{N}, i_1, \dots, i_s \in I, u \in M, n_1 + \dots + n_s \ge p\}.
\end{equation*}

\begin{proposition}[{\cite{li_abelianizing_2005}}]
  \label{prp:5}
  The filtration $(F_pM)_{p \in \mathbb{Z}}$ satisfies:
  \begin{enumerate}
  \item $M = F_pM$ for $p \le 0$;
  \item $F_0M \supseteq F_1M \supseteq \dots$;
  \item $a_{(n)}F_qM \subseteq F_{p + q - n - 1}M$ for $p, q \in \mathbb{Z}$, $a \in F_pV$ and $n \in \mathbb{Z}$;
  \item $a_{(n)}F_qM \subseteq F_{p + q - n}M$ for $p, q \in \mathbb{Z}$, $a \in F_pV$ and $n \in \mathbb{N}$.
  \end{enumerate}
\end{proposition}

Let
\begin{equation*}
  \gr_F(M) = \bigoplus_{p \in \mathbb{N}}F_pM/F_{p + 1}M
\end{equation*}
be the associated graded vector space.
By \cite{li_abelianizing_2005}, the vector space $\gr_F(M)$ is a module over $\gr_F(V)$ with operations given as follows.
For $p, q \in \mathbb{N}$, $a \in F_pV$ and $u \in F_qM$, we set:
\begin{align*}
  \sigma_p(a)\sigma^M_q(u) &= \sigma^M_{p + q}(a^M_{(-1)}u), \\
  Y^M_-(\sigma_p(a), z)\sigma^M_q(u) &= \sum_{n \in \mathbb{N}}\sigma^M_{p + q - n}(a^M_{(n)}u)z^{-n - 1},
\end{align*}
where $\sigma^M_p: F_pM \to \gr_F(M)$ is the principal symbol map.
The filtration $(F_pM)_{p \in \mathbb{Z}}$ is called the Li filtration of $M$.

If $V$ is graded with Hamiltonian $H$, and $M$ is graded with Hamiltonian $H^M$, then $H^M(F_pM) \subseteq F_pM$ for $p \in \mathbb{Z}$.
Therefore, we can define an operator $H^M \in \End(\gr_F(M))$ as $H^M(\sigma^M_p(u)) = \sigma^M_p(H^Mu)$ for $p \in \mathbb{N}$ and $u \in F_pM$.
For $p \in \mathbb{Z}$ and $\Delta \in \mathbb{C}$, we define $F_pM_\Delta = F_pM \cap M_\Delta$.
Since $H^M(F_pM) \subseteq F_pM$ for $p \in \mathbb{Z}$, by \cite[Corollary 1.1]{kac_bombay_2013}, we have
\begin{equation*}
  F_pM = \bigoplus_{\Delta \in \mathbb{C}}F_pM_\Delta \quad \text{for $p \in \mathbb{Z}$}.
\end{equation*}
For $\Delta \in \mathbb{C}$, we define $\gr_F(M)_\Delta = \bigoplus_{p \in \mathbb{N}}\sigma^M_p(F_pM_\Delta)$.
Then $H^Mu = \Delta u$ for $u \in \gr_F(M)_\Delta$.
The family of subspaces $(\gr_F(M)_\Delta)_{\Delta \in \mathbb{C}}$ satisfies $\gr_F(M) = \bigoplus_{\Delta \in \mathbb{C}}\gr_F(M)_\Delta$.
Therefore, the operator $H^M \in \End(\gr_F(M))$ is diagonalizable with $\gr_F(M)_\Delta = \ker(H^M - \Delta\Id_{\gr_F(M)})$.
In fact, it is possible to prove that this diagonalizable operator $H^M$ is actually a Hamiltonian of $\gr_F(M)$.

The following proposition is similar to \zcref{prp:1}, but it requires an $h + \mathbb{N}$-graded $V$-module $M$ and deals with the decreasing filtration $(F_pM)_{p \in \mathbb{Z}}$.

\begin{proposition}
  \label{prp:6}
  Let $V$ be an $\mathbb{N}$-graded vertex algebra, let $M$ be an $h + \mathbb{N}$-graded $V$-module, and let $(B_p)_{p \in \mathbb{N}}$ be a decreasing family of sets such that $B_p \subseteq F_pM$, and all elements of $B_p$ are homogeneous for $p \in \mathbb{N}$.
  Then $\{v + F_{p + 1}M \mid v \in B_p \setminus B_{p + 1}\}$ spans (resp.\ is a basis of) $F_pM/F_{p + 1}M$ for $p \in \mathbb{N}$ if and only if $B_p$ spans (resp.\ is a basis of) $F_pM$ for $p \in \mathbb{N}$.
  In particular, if $\{\sigma_p^M(v) \mid p \in \mathbb{N}, v \in B_p \setminus B_{p + 1}\}$ is a basis of $\gr_F(M)$, then $B_0$ is a basis of $M$.
\end{proposition}

\begin{example}[$\gr_F(M(c, h))$]
  \label{exa:5}
  We pick a highest weight $(c, h)$.
  From \zcref{exa:4}, $\gr_F(\Vir^c)$ is isomorphic to $\mathbb{C}[L_{-2}, L_{-3}, \dots]$.
  We have an isomorphism
  \begin{align*}
    \gr_F(M(c, h)) &\xrightarrow{\sim} \bigoplus_{k \in \mathbb{N}}\mathbb{C}[L_{-2}, L_{-3}, \dots]L_{-1}^k, \\
    \sigma_{\Delta(\lambda) - \clen(\lambda)}(L_\lambda|c, h\rangle) &\mapsto u_\lambda \quad \text{for $\lambda$ a composition}.
  \end{align*}
  The Hamiltonian is given by $L_0$.
\end{example}

\section{Classically free boundary minimal models}
\label{sec:class-free-bound}

Let $V$ be a vertex algebra.
By \cite[Lemma 2.9]{li_abelianizing_2005}, we have
\begin{equation*}
  C_2V = \vspan\{a_{(-2)}b \mid a, b \in V\} = F_1V.
\end{equation*}
We define the \emph{Zhu $C_2$-algebra of $V$} by
\begin{equation*}
  R_V = V/C_2V = F_0V/F_1V \subseteq \gr_F(V).
\end{equation*}

The fact that $\gr_F(V)$ is a vertex Poisson algebra implies that $R_V$ is a Poisson algebra with operations given as follows.
For $a, b \in V$, we set:
\begin{align*}
  \sigma_0(a)\sigma_0(b) &= \sigma_0(a_{(-1)}b), \\
  \{\sigma_0(a), \sigma_0(b)\} &= \sigma_0(a_{(0)}b).
\end{align*}

If $V$ is graded, then, as we explained in \zcref{sec:li-filtr}, $\gr_F(V)$ is graded.
Thus, $R_V$ is also graded.

We have constructed a functor
\begin{equation*}
  R: \{\text{(graded) vertex algebras}\} \to \{\text{(graded) Poisson algebras}\}.
\end{equation*}

Oftentimes, some condition on $R_V$ implies or is equivalent to some condition on $V$.
The vertex algebra $V$ is called \emph{$C_2$-cofinite} if $R_V$ is finite dimensional.

\begin{example}[$R_{\Vir^c}$ and $R_{\Vir_c}$]
  \label{exa:6}
  We pick $c \in \mathbb{C}$.
  We consider $\mathbb{C}[L_{-2}]$ as the polynomial algebra in one variable $L_{-2}$, and we equip it with the trivial Poisson bracket.
  By \zcref{exa:4}, we have the following isomorphism of Poisson algebras
  \begin{align*}
    R_{\Vir^c} &\xrightarrow{\sim} \mathbb{C}[L_{-2}], \\
    \sigma_0(L_{-2}\vac) &\mapsto L_{-2}.
  \end{align*}

  We now move on to $\Vir_c$.
  If $c$ is not of the form $c_{p, q}$ for some $p, q \ge 2$ relatively prime integers, then $\Vir_c = \Vir^c$ by \zcref{thr:4}, and we have already solved the problem.
  Therefore, we assume $c$ is of this form.

  We have a natural quotient map
  \begin{align*}
    \pi_c: \Vir^c &\twoheadrightarrow \Vir_c, \\
    \pi_c(a) &= a + U(\Vir)\{a_{p, q}\}.
  \end{align*}
  Applying the functor $R$, we obtain an epimorphism
  \begin{equation*}
    R_{\pi_c}: R_{\Vir^c} \twoheadrightarrow R_{\Vir_c}.
  \end{equation*}
  From the equation $\ker(R_{\pi_c}) = \sigma_0(U(\Vir)\{a_{p, q}\})$ and \zcref{thr:4}, we obtain
  \begin{equation*}
    \ker(R_{\pi_c}) = (\sigma_0(L_{-2}^{(p - 1)(q - 1)/2}\vac)).
  \end{equation*}
  In summary,
  \begin{equation*}
    R_{\Vir_{p, q}} \cong \mathbb{C}[L_{-2}]/(L_{-2}^{(p - 1)(q - 1)/2}).
  \end{equation*}
  Thus, $\Vir^c$ is never $C_2$-cofinite, while $\Vir_c$ is $C_2$-cofinite only when $c$ is of the form $c_{p, q}$ for some $p, q \ge 2$ relatively prime integers.
\end{example}

\begin{lemma}[{\cite[Corollary 4.3]{li_abelianizing_2005}}]
  \label{lmm:7}
  Let $V$ be a vertex algebra.
  As a differential algebra, $\gr_F(V)$ is generated by $R_V$, i.e.,
  \begin{equation*}
    \gr_F(V) = (R_V)_T.
  \end{equation*}
\end{lemma}

Let $V$ be a (graded) vertex algebra.
We have a natural (graded) algebra inclusion $\inc: R_V \hookrightarrow \gr_F(V)$.
By the universal property of $\inc: R_V \hookrightarrow JR_V$ (see \zcref{sec:graded-jet-algebras}), there is a (graded) differential algebra homomorphism $\phi_V: JR_V \to \gr_F(V)$ such that the following diagram commutes
\begin{equation*}
  \begin{tikzcd}
    R_V \arrow[rd, "\inc"', hook] \arrow[r, "\inc", hook] & JR_V \arrow[d, "\phi_V"] \\
    & \gr_F(V)
  \end{tikzcd}
\end{equation*}
Because $R_V$ is a (graded) Poisson algebra, we can equip $JR_V$ with the level 0 vertex Poisson algebra structure, as explained in \zcref{sec:graded-vert-poiss}.
From now on, $JR_V$ will be considered as a (graded) vertex Poisson algebra.

\begin{lemma}
  \label{lmm:8}
  Let $V$ be a (graded) vertex algebra.
  The (graded) differential algebra homomorphism $\phi_V: JR_V \to \gr_F(V)$ defined above is surjective and is actually a (graded) vertex Poisson algebra homomorphism.
\end{lemma}

\begin{proof}
  The homomorphism $\phi_V$ is surjective by \zcref{lmm:7}.
  The fact that $\phi_V$ is a (graded) vertex Poisson algebra homomorphism is explained in \cite[Proposition 2.5.1]{arakawa_remark_2012}.
\end{proof}

When $\phi_V$ is an isomorphism, we say $V$ is \emph{classically free}.

\begin{example}[$JR_{\Vir^c}$ and $JR_{\Vir_c}$]
  \label{exa:7}
  Let $c \in \mathbb{C}$.
  Then $\Vir^c$ is classically free because
  \begin{equation*}
    \gr_F(\Vir^c) \cong \mathbb{C}[L_{-2}, L_{-3}, \dots]
  \end{equation*}
  by \zcref{exa:4} and
  \begin{equation*}
    JR_{\Vir^c} \cong J(\mathbb{C}[L_{-2}]) = \mathbb{C}[L_{-2}, L_{-3}, \dots]
  \end{equation*}
  by \zcref{exa:6} and \zcref{sec:graded-jet-algebras}.

  We now move on to $\Vir_c$.
  If $c$ is not of the form $c_{p, q}$ for some $p, q \ge 2$ relatively prime integers, then $\Vir_c = \Vir^c$ by \zcref{thr:4}, and we have already solved the problem.
  Therefore, we assume $c$ is of this form.
  Then
  \begin{equation*}
    JR_{\Vir_{p, q}} \cong J(\mathbb{C}[L_{-2}]/(L_{-2}^{(p - 1)(q - 1)/2})) = \mathbb{C}[L_{-2}, L_{-3}, \dots]/(L_{-2}^{(p - 1)(q - 1)/2})_\partial
  \end{equation*}
  by \zcref{exa:6} and \zcref{sec:graded-jet-algebras}.
\end{example}

\begin{example}
  \label{exa:8}
  If $q > p \ge 2$ are relatively prime integers, then $\Vir_{p, q}$ is classically free if and only if $p = 2$ by \cite{van_ekeren_chiral_2021}.
\end{example}

\begin{example}
  \label{exa:9}
  It was proven in \cite{andrews_singular_2022} that the Ising model $\Vir_{3, 4}$ is not classically free.
  In fact, by \cite[Theorem 2]{andrews_singular_2022},
  \begin{equation*}
    \ker(\phi_{\Vir_{3, 4}}) = (b)_\partial,
  \end{equation*}
  where
  \begin{equation*}
    b = L_{-4}L_{-3}L_{-2} + \tfrac{1}{6}L_{-5}L_{-2}^2,
  \end{equation*}
  and $(b)_\partial$ is the differential ideal generated by $b$, cf.\ \zcref{sec:graded-jet-algebras} and \zcref{exa:4} where $\partial$ is denoted by $T$.
\end{example}

Let $V$ be a vertex algebra, and let $M$ be a module over $V$.
By \cite[Lemma 2.9]{li_abelianizing_2005}, we have
\begin{equation*}
  C_2M = \vspan\{a^M_{(-2)}u \mid a \in V, u \in M\} = F_1M.
\end{equation*}
We define the \emph{Zhu $C_2$-module of $M$} by
\begin{equation*}
  R_M = M/C_2M = F_0M/F_1M \subseteq \gr_F(M).
\end{equation*}
The fact that $\gr_F(M)$ is a module over $\gr_F(V)$ implies that $R_M$ is a module over $R_V$ with operations given as follows.
For $a \in V$ and $u \in M$, we set:
\begin{align*}
  \sigma_0(a)\sigma^M_0(u) &= \sigma^M_0(a^M_{(-1)}u), \\
  \{\sigma_0(a), \sigma^M_0(u)\} &= \sigma^M_0(a^M_{(0)}u).
\end{align*}

If $V$ is graded and $M$ is a graded $V$-module, then, as we explained in \zcref{sec:li-filtr}, $\gr_F(M)$ is a graded $\gr_F(V)$-module.
Thus, $R_M$ is also a graded $R_V$-module.

\begin{example}[$R_{M(c, h)}$ and $R_{L(c, h)}$]
  \label{exa:10}
  We pick a highest weight $(c, h)$.
  As in \zcref{exa:6}, we consider $\mathbb{C}[L_{-2}]$ as the polynomial algebra in one variable $L_{-2}$, and we equip it with the trivial Poisson bracket.
  We consider $\bigoplus_{k \in \mathbb{N}}\mathbb{C}[L_{-2}]L_{-1}^k$ as a module over the Poisson algebra $\mathbb{C}[L_{-2}]$ with Poisson bracket given by $\{L_{-2}, L_{-1}^k\} = L_{-1}^{k + 1}$ for $k \in \mathbb{N}$.
  By \zcref{exa:5} and \zcref{exa:6}, we have the following isomorphism of modules over Poisson algebras
  \begin{align*}
    R_{M(c, h)} &\xrightarrow{\sim} \bigoplus_{k \in \mathbb{N}}\mathbb{C}[L_{-2}]L_{-1}^k, \\
    \sigma_0(L_{-2}|c, h\rangle) &\mapsto L_{-2}, \\
    \sigma_0(L_{-1}|c, h\rangle) &\mapsto L_{-1}.
  \end{align*}
  As in \zcref{exa:6}, we obtain
  \begin{equation*}
    R_{L(c, h)} \cong \frac{\bigoplus_{k \in \mathbb{N}}\mathbb{C}[L_{-2}]L_{-1}^k}{\sigma_0(J(c, h))}.
  \end{equation*}
\end{example}

\begin{theorem}[{\cite[Lemma 4.2]{li_abelianizing_2005}}]
  \label{thr:6}
  Let $V$ be a vertex algebra, and let $M$ be a $V$-module.
  As a $\gr_F(V)$-module without the Poisson structure, $\gr_F(M)$ is generated by $R_M$, i.e.,
  \begin{equation*}
    \gr_F(M) = \gr_F(V)R_M.
  \end{equation*}
\end{theorem}

Let $V$ be a (graded) vertex algebra, and let $M$ be a (graded) $V$-module.
We have a natural (graded) inclusion $\inc: R_M \hookrightarrow \gr_F(M)$ of (graded) $R_V$-modules.
As we explained in \zcref{sec:graded-vert-poiss}, we can consider $JR_V \otimes_{R_V} R_M$ as a (graded) $JR_V$-module.
Because we have a (graded) epimorphism of vertex Poisson algebras $\phi_V: JR_V \twoheadrightarrow \gr_F(V)$, we can consider the (graded) $\gr_F(V)$-module $\gr_F(M)$ as a (graded) $JR_V$-module.
By the universal property of $\inc: R_M \hookrightarrow JR_V \otimes_{R_V} R_M$ (see \zcref{sec:graded-vert-poiss}), there is a (graded) $JR_V$-module homomorphism $\phi_M: JR_V \otimes_{R_V} R_M \to \gr_F(M)$ such that the following diagram commutes
\begin{equation*}
  \begin{tikzcd}
    R_M \arrow[r, "\inc", hook] \arrow[rd, "\inc"', hook] & JR_V \otimes_{R_V} R_M \arrow[d, "\phi_M"] \\
    & \gr_F(M)
  \end{tikzcd}
\end{equation*}

\begin{lemma}
  \label{lmm:9}
  The (graded) $JR_V$-module homomorphism $\phi_M: JR_V \otimes_{R_V} R_M \to \gr_F(M)$ defined above is surjective.
\end{lemma}

\begin{proof}
  The assertion follows from \zcref{lmm:8} and \zcref{thr:6}.
  This is also explained in a slightly more general setting in \cite{arakawa_associated_2015}.
\end{proof}

When $\phi_M$ is an isomorphism, we say $M$ is \emph{classically free}.

\begin{example}[$JR_{\Vir^c} \otimes_{R_{\Vir^c}} R_{M(c, h)}$]
  \label{exa:11}
  We pick a highest weight $(c, h)$.
  The Verma module $M(c, h)$ is always classically free as a $\Vir^c$-module because by \zcref{exa:5}, \zcref{exa:6} and \zcref{exa:10},
  \begin{align*}
    JR_{\Vir^c} \otimes_{R_{\Vir^c}} R_{M(c, h)} &\cong \mathbb{C}[L_{-2}, L_{-3}, \dots] \otimes_{\mathbb{C}[L_{-2}]} \bigoplus_{k \in \mathbb{N}}\mathbb{C}[L_{-2}]L_{-1}^k \\
                                           &\cong \bigoplus_{k \in \mathbb{N}}\mathbb{C}[L_{-2}, L_{-3}, \dots]L_{-1}^k \\
                                           &\cong \gr_F(M(c, h)).
  \end{align*}
\end{example}

We wish to determine when the irreducible modules $L(c_{p, q}, h_{m, n})$ over the simple Virasoro vertex algebras $\Vir_{p, q}$ are classically free, as mentioned in \zcref{thr:3} and \zcref{thr:5}.

\begin{proof}[Proof of \zcref{thr:3}]
  By \zcref{exa:10}, we have
  \begin{equation*}
    R_{L(c_{2, 2s + 1}, h_{1, i})} \cong \frac{\bigoplus_{k \in \mathbb{N}}\mathbb{C}[L_{-2}]L_{-1}^k}{\sigma_0(J(c_{2, 2s + 1}, h_{1, i}))}.
  \end{equation*}
  We use \zcref{prp:6} with the basis $B_0$ provided by \zcref{thr:2} and
  \begin{equation*}
    B_p = \{L_{\lambda}|c_{2, 2s + 1}, h_{1, i}\rangle \mid \lambda \in P^{s, i}, \Delta(\lambda) - \clen(\lambda) \ge p\} \quad \text{for $p \in \mathbb{N}$}.
  \end{equation*}
  Therefore, the set
  \begin{equation}
    \label{eq:9}
    \{u_{\lambda} + \sigma_0(J(c_{2, 2s + 1}, h_{1, i})) \mid \lambda \in P^{s, i}, \Delta(\lambda) - \clen(\lambda) = 0\}
  \end{equation}
  is a basis of $R_{L(c_{2, 2s + 1}, h_{1, i})}$.
  We note that for a partition $\lambda$, $\Delta(\lambda) - \clen(\lambda) = 0$ is equivalent to $\lambda = \emptyset$ or $2 \ge \lambda_1$.

  By \cite{van_ekeren_chiral_2021} or \cite{bruschek_arc_2013}, we already know $\Vir_{2, 2s + 1}$ is classically free.
  This means that $\gr_F(\Vir_{2, 2s + 1}) \cong JR_{\Vir_{2, 2s + 1}} \cong J(\mathbb{C}[L_{-2}]/(L_{-2}^s)) = \mathbb{C}[L_{-2}, L_{-3}, \dots]/(L_{-2}^s)_\partial$ and
  \begin{equation}
    \label{eq:10}
    \{p_\lambda + (L_{-2}^s)_\partial \mid \lambda \in P^{s, 1}\}
  \end{equation}
  is a basis of $JR_{\Vir_{2, 2s + 1}}$.

  We have
  \begin{equation}
    \label{eq:11}
    JR_{\Vir_{2, 2s + 1}} \otimes_{R_{\Vir_{2, 2s + 1}}} R_{L(c_{2, 2s + 1}, h_{1, i})} \cong \frac{\mathbb{C}[L_{-2}, L_{-3}, \dots]}{(L_{-2}^s)_\partial} \otimes_{\frac{\mathbb{C}[L_{-2}]}{(L_{-2}^s)}}\frac{\bigoplus_{k \in \mathbb{N}}\mathbb{C}[L_{-2}]L_{-1}^k}{\sigma_0(J(c_{2, 2s + 1}, h_{1, i}))}.
  \end{equation}

  Given a partition $\lambda \in P^{s, i}$, we can write $\lambda$ uniquely as $\lambda = [\lambda^1, \lambda^2]$, where $\lambda^1 \in P^{s, 1}$ satisfies $\lambda^1 = \emptyset$ or $\lambda^1_{\len(\lambda^1)} \ge 3$ and $\lambda^2 = \emptyset$ or ($\lambda^2 \in P^{s, i}$ and $2 \ge \lambda^2_1$).

  Given $\lambda \in P^{s, i}$, we define
  \begin{equation*}
    t_\lambda = (p_{\lambda^1} + (L_{-2}^s)_\partial)\otimes(p_{\lambda^2} + \sigma_0(J(c_{2, 2s + 1}, h_{1, i}))) \in \frac{\mathbb{C}[L_{-2}, L_{-3}, \dots]}{(L_{-2}^s)_\partial} \otimes_{\frac{\mathbb{C}[L_{-2}]}{(L_{-2}^s)}} \frac{\bigoplus_{k \in \mathbb{N}}\mathbb{C}[L_{-2}]L_{-1}^k}{\sigma_0(J(c_{2, 2s + 1}, h_{1, i}))}.
  \end{equation*}
  From the basis \eqref{eq:9} of $R_{L(c_{2, 2s + 1}, h_{1, i})}$, the basis \eqref{eq:10} of $JR_{\Vir_{2, 2s + 1}}$ and \eqref{eq:11}, we conclude that
  \begin{equation}
    \label{eq:12}
    JR_{\Vir_{2, 2s + 1}} \otimes_{R_{\Vir_{2, 2s + 1}}} R_{L(c_{2, 2s + 1}, h_{1, i})} = \vspan\{t_\lambda \mid \lambda \in P^{s, i}\}.
  \end{equation}
  By \zcref{thr:2} and \zcref{prp:6}, the natural epimorphism $\phi_{L(c_{2, 2s + 1}, h_{1, i})}$ maps the set $\{t_\lambda \mid \lambda \in P^{s, i}\}$ onto a basis of $\gr_F(L(c_{2, 2s + 1}, h_{1, i}))$.
  Thus, the set $\{t_\lambda \mid \lambda \in P^{s, i}\}$ is linearly independent, and by \eqref{eq:12}, it is a basis of $JR_{\Vir_{2, 2s + 1}} \otimes_{R_{\Vir_{2, 2s + 1}}} R_{L(c_{2, 2s + 1}, h_{1, i})}$.
  Therefore, $\phi_{L(c_{2, 2s + 1}, h_{1, i})}$ is an isomorphism, and $L(c_{2, 2s + 1}, h_{1, i})$ is classically free.

  Let $p, q > 2$ be relatively prime integers, and let $m, n$ be integers such that $0 < m < p$ and $0 < n < q$.
  We have two natural inclusions $\inc: JR_{\Vir_{p, q}} \hookrightarrow JR_{\Vir_{p, q}} \otimes_{R_{\Vir_{p, q}}} R_{L(c_{p, q}, h_{m, n})}, \inc(a) = a\otimes1$ and $\inc: \gr_F(\Vir_{p, q}) \hookrightarrow \gr_F(L(c_{p, q}, h_{m, n})), \inc(a) = a1$.
  We assume, for the sake of contradiction, that $\phi_{L(c_{p, q}, h_{m, n})}$ is an isomorphism.
  This would imply that $\phi_{\Vir_{p, q}}$ is also an isomorphism, a contradiction to \zcref{exa:8}.
  Finally, $\phi_{L(c_{p, q}, h_{m, n})}$ is not an isomorphism, and $L(c_{p, q}, h_{m, n})$ is not classically free as a module over $\Vir_{p, q}$.
\end{proof}

As in the proof of \zcref{lmm:5}, let $w \in J(c_{2, 2s + 1}, h_{1, i})$ be the singular vector such that $\lp(w) = L_{-1}^i|c_{2, 2s + 1}, h_{1, i}\rangle$, and we set:
\begin{align*}
  u_k &= \sigma_0((a_{2, 2s + 1})_{(k - 1)}|c_{2, 2s + 1}, h_{1, i}\rangle) \quad \text{for $k = 1, \dots, i - 1$}, \\
  u_i &= \sigma_0(w).
\end{align*}
Then $u_1, \dots, u_i \in \sigma_0(J(c_{2, 2s + 1}, h_{1, i}))$ by \zcref{thr:5}.
Moreover:
\begin{align*}
  L_{-2}^s &= \sigma_0(a_{2, 2s + 1})\sigma_0(|c_{2, 2s + 1}, h_{1, i}\rangle) = \sigma_0((a_{2, 2s + 1})_{(-1)}|c_{2, 2s + 1}, h_{1, i}\rangle) \in \sigma_0(J(c_{2, 2s + 1}, h_{1, i})), \\
  \lp(u_k) &= L_{-2}^{s - k}L_{-1}^k \quad \text{for $k = 1, \dots, i - 1$}, \\
  \lp(u_i) &= L_{-1}^i.
\end{align*}

\begin{corollary}
  \label{crl:1}
  With the notation above, we have:
  \begin{align}
    \label{eq:13}
    \sigma_0(J(c_{2, 2s + 1}, h_{1, i})) &= (L_{-2}^s, u_1, \dots, u_i)_\psn, \\
    \label{eq:14}
    R_{L(c_{2, 2s + 1}, h_{1, i})} &\cong \frac{\bigoplus_{k \in \mathbb{N}}\mathbb{C}[L_{-2}]L_{-1}^k}{(L_{-2}^s, u_1, \dots, u_i)_\psn}, \\
    \label{eq:15}
    \gr_F(L(c_{2, 2s + 1}, h_{1, i})) &\cong \frac{\mathbb{C}[L_{-2}, L_{-3}, \dots]}{(L_{-2}^s)_\partial} \otimes_{\frac{\mathbb{C}[L_{-2}]}{(L_{-2}^s)}} \frac{\bigoplus_{k \in \mathbb{N}}\mathbb{C}[L_{-2}]L_{-1}^k}{(L_{-2}^s, u_1, \dots, u_i)_\psn},
  \end{align}
  where the subscript $\psn$ denotes the Poisson submodule generated by the given subset.
\end{corollary}

\begin{proof}[Sketch of proof]
  We use the notation of the proof of \zcref{thr:3}.
  Let $K' = (L_{-2}^s, u_1, \dots, u_i)_\psn \subseteq \sigma_0(J(c_{2, 2s + 1}, h_{1, i}))$ be the graded submodule of $\bigoplus_{k \in \mathbb{N}}\mathbb{C}[L_{-2}]L_{-1}^k$.
  We note that for a partition $\lambda \in R^{s, i}$ with $\lambda = \emptyset$ or $2 \ge \lambda_1$, there is $u \in K'$ such that $\lp(u) = u_\lambda$.
  We can prove that $K' = \sigma_0(J(c_{2, 2s + 1}, h_{1, i}))$ using the argument of (the proof of) \cite[Corollary 4.1]{salazar_pbw_2024} comparing the dimensions of the graded components of $\bigoplus_{k \in \mathbb{N}}\mathbb{C}[L_{-2}]L_{-1}^k/\sigma_0(J(c_{2, 2s + 1}, h_{1, i}))$ and $\bigoplus_{k \in \mathbb{N}}\mathbb{C}[L_{-2}]L_{-1}^k/K'$ up to $N \in \mathbb{N}$, and then we let $N \to \infty$ to conclude.
  The isomorphisms \eqref{eq:14} and \eqref{eq:15} follow.
\end{proof}

\begin{remark}
  \label{rmk:3}
  We have
  \begin{equation*}
    u_1 = \sigma_0((a_{2, 2s + 1})_{(0)}|c_{2, 2s + 1}, h_{1, i}\rangle) = \{\sigma_0(a_{2, 2s + 1}), \sigma_0(|c_{2, 2s + 1}, h_{1, i}\rangle)\} = \{L_{-2}^s, 1\} = sL_{-2}^{s - 1}L_{-1}.
  \end{equation*}
  However, I think there is no simple expression for $u_2, \dots, u_i$ and $R_{L(c_{2, 2s + 1}, h_{1, i})}$.
\end{remark}

\appendix
\section{Graded Poisson algebras and their modules}
\label{sec:grad-poiss-algeb}

Let $A$ be a commutative associative algebra.
A \emph{Hamiltonian operator of $A$} is a diagonalizable operator $H \in \End(A)$ such that
\begin{equation*}
  H(ab) = H(a)b + aH(b) \quad \text{for $a, b \in A$}.
\end{equation*}
Thus, a Hamiltonian of $A$ is just a derivation of $A$.
An algebra with a Hamiltonian is called \emph{graded}.

Let $A$ be a differential commutative associative algebra with derivation $\partial$.
A \emph{Hamiltonian operator of $A$} is a Hamiltonian of $A$ as a commutative associative algebra such that
\begin{equation*}
  [H, \partial] = \partial.
\end{equation*}
It is possible to prove inductively that the last equation implies
\begin{equation}
  \label{eq:16}
  H\partial^n = n\partial^n + \partial^nH \quad \text{for $n \in \mathbb{N}$}.
\end{equation}

Let $A$ be a Poisson algebra as defined in \cite{caressa_examples_2003}.
A \emph{Hamiltonian operator of $A$} is a Hamiltonian $H$ of $A$ as a commutative associative algebra such that
\begin{equation*}
  H(\{a, b\}) = \{a, H(b)\} + \{H(a), b\} - \{a, b\} \quad \text{for $a, b \in A$}.
\end{equation*}

Let $A$ be a graded commutative associative algebra with Hamiltonian $H$, and let $M$ be an $A$-module.
If $A$ has a unit $1$, we further assume that $1u = u$ for $u \in M$.
A \emph{Hamiltonian operator of $M$} is a diagonalizable operator $H^M \in \End(M)$ such that
\begin{equation*}
  H^M(au) = H(a)u + aH^M(u) \quad \text{for $a \in A$ and $u \in M$}.
\end{equation*}

Let $A$ be a Poisson algebra.
A \emph{module over $A$} is an $A$-module $M$ in the usual associative sense equipped with a bilinear map $\{\bullet, \bullet\}: A \times M \to M$, which makes $M$ a Lie algebra module over $A$ such that for $a, b \in A$ and $u \in M$:
\begin{enumerate}
\item (Left Leibniz rule) $\{a, bu\} = \{a, b\}u + b\{a, u\}$;
\item (Right Leibniz rule) $\{ab, u\} = a\{b, u\} + b\{a, u\}$.
\end{enumerate}

Let $A$ be a graded Poisson algebra with Hamiltonian $H$, and let $M$ be an $A$-module.
A \emph{Hamiltonian operator of $M$} is a Hamiltonian $H^M$ of $M$ as a module over $A$ as a commutative associative algebra such that
\begin{equation*}
  H^M(\{a, u\}) = \{a, H^M(u)\} + \{H(a), u\} - \{a, u\} \quad \text{for $a \in A$ and $u \in M$}.
\end{equation*}

\section{Graded Jet algebras and their modules}
\label{sec:graded-jet-algebras}

In this appendix, by an algebra we will mean a commutative associative algebra with unit.
Let $R$ be a finitely generated algebra.
We now construct a differential algebra $JR$ called the \emph{jet algebra of $R$} and an algebra inclusion $\inc: R \hookrightarrow JR$ universal with this property, i.e., for a differential algebra $A$ and an algebra homomorphism $\phi: R \to A$, there is a unique differential algebra homomorphism $\overline{\phi}: JR \to A$ such that the following diagram commutes
\begin{equation*}
  \begin{tikzcd}
    R \arrow[r, "\inc", hook] \arrow[rd, "\phi"'] & JR \arrow[d, "\overline{\phi}"] \\
    & A
  \end{tikzcd}
\end{equation*}

Assuming $R = \mathbb{C}[x^1, \dots, x^r]/(f_1, \dots, f_s)$ for some polynomials $f_1, \dots, f_s \in \mathbb{C}[x^1, \dots, x^r]$, the construction is as follows.
We introduce new variables $x^j_{(-i)}$ for $i = 1, 2, \dots$, $j = 1, \dots, r$ and a derivation $\partial$ of the polynomial algebra $\mathbb{C}[x^j_{(-i)} \mid i = 1, 2, \dots, j = 1, \dots, r]$ by setting
\begin{equation*}
  \partial x^j_{(-i)} = ix^j_{(-i - 1)} \quad \text{for $i = 1, 2, \dots$, $j = 1, \dots, r$}.
\end{equation*}
We set (identifying $x^j$ with $x^j_{(-1)}$ when considering $f_i$ in the following equation)
\begin{align*}
  JR &= \mathbb{C}[x^j_{(-i)} \mid i = 1, 2, \dots, j = 1, \dots, r]/(\partial^jf_i \mid i = 1, \dots, s, j = 0, 1, \dots) \\
     &= \mathbb{C}[x^j_{(-i)} \mid i = 1, 2, \dots, j = 1, \dots, r]/(f_1, \dots, f_s)_\partial,
\end{align*}
where the subscript $\partial$ indicates the differential ideal generated by the given subset.
By our definitions, $\partial$ factors through a derivation of $JR$, and we have an algebra inclusion $\inc: R \hookrightarrow JR, \inc(x^j + (f_1, \dots, f_s)) = x^j_{(-1)} + (f_1, \dots, f_s)_\partial$ for $j = 1, \dots, r$.
The fact that $\inc: R \hookrightarrow JR$ satisfies our desired universal property is explained in \cite[\S2.3]{arakawa_remark_2012} and \cite{ein_jet_2008}.

\begin{remark}
  \label{rmk:4}
  We see that the classes of the original variables $x^j$ generate $JR$ as a differential algebra, i.e.,
  \begin{equation*}
    JR = (x^1 + (f_1, \dots, f_s)_\partial, \dots, x^r + (f_1, \dots, f_s)_\partial)_\partial.
  \end{equation*}
\end{remark}

If $f: R_1 \to R_2$ is a homomorphism of finitely generated algebras, then $Jf: JR_1 \to JR_2$ is defined by requiring that $Jf(\partial_1^n(\inc_1(x))) = \partial^n_2(\inc_2(f(x)))$ for $x \in R_1$ and $n \in \mathbb{N}$ (cf.\ \zcref{rmk:4}).
For a finitely generated algebra $R$ and a differential algebra $A$, we have a natural isomorphism
\begin{equation*}
  \Hom_{\{\text{differential algebras}\}}(JR, A) \cong \Hom_{\{\text{algebras}\}}(R, A).
\end{equation*}

\begin{remark}
  \label{rmk:5}
  It is not true that we have defined a functor $J$ which is left adjoint to the forgetful functor $\{\text{differential algebras}\} \to \{\text{finitely generated algebras}\}$ because a differential algebra is generally not finitely generated.

  But we can easily work in the general case as follows.
  Let $R$ be an algebra (not necessarily finitely generated), and we consider the polynomial algebra in $R$ variables $\mathbb{C}[(x^j)_{j \in R}]$.
  We have a natural epimorphism $\pi: \mathbb{C}[(x^j)_{j \in R}] \twoheadrightarrow R, \pi(x^j) = j$ for $j \in R$.
  We can repeat the construction of $JR$ with $\mathbb{C}[(x^j)_{j \in R}]$ in place of $\mathbb{C}[x^1, \dots, x^r]$ and $\ker(\pi)$ in place of $(f_1, \dots, f_s)$.
  This way, we construct a functor
  \begin{equation*}
    J: \{\text{algebras}\} \to \{\text{differential algebras}\},
  \end{equation*}
  which is left adjoint to the forgetful functor $\{\text{differential algebras}\} \to \{\text{algebras}\}$.
\end{remark}

Given an algebra $R$, it is useful to consider the functor
\begin{equation*}
  JR \otimes_R \bullet: \text{$R$-Mod} \to \text{$JR$-Mod},
\end{equation*}
where $JR$ is merely considered an algebra.
Again, the $JR$-module $M$ together with the $R$-module inclusion $\inc: M \hookrightarrow JR \otimes_R M, \inc(u) = 1\otimes u$ satisfy a universal property similar to that of $\inc: R \hookrightarrow JR$, and the functor $JR \otimes_R \bullet$ is left adjoint to the forgetful functor $\text{$JR$-Mod} \to \text{$R$-Mod}$.

Let $R$ be a graded algebra with Hamiltonian $H$.
We can extend uniquely the Hamiltonian $H$ to a Hamiltonian $H^{JR} \in \End(JR)$ because of \eqref{eq:16}.
Furthermore, if $M$ is a graded $R$-module with Hamiltonian $H^M$, we can define a Hamiltonian of $JR \otimes_R M$ by setting
\begin{equation*}
  H^{JR \otimes_R M} = H^{JR}\otimes\Id_M + \Id_{JR}\otimes H^M.
\end{equation*}

\section{Graded vertex Poisson algebras and their modules}
\label{sec:graded-vert-poiss}

\begin{proposition}[{\cite{arakawa_associated_2015}, \cite[Proposition 2.3.1]{arakawa_remark_2012} and \zcref{sec:graded-jet-algebras}}]
  \label{prp:7}
  Let $R$ be a (graded) Poisson algebra.
  Then there is a unique (graded) vertex Poisson algebra structure on $JR$ such that
  \begin{equation*}
    a_{(n)}b = \delta_{n, 0}\{a, b\} \quad \text{for $a, b \in R$ and $n \in \mathbb{N}$}.
  \end{equation*}
\end{proposition}

The (graded) vertex Poisson algebra structure on $JR$ given in \zcref{prp:7} for a (graded) Poisson algebra $R$ will be called the \emph{level 0 vertex Poisson algebra structure of $JR$}.

If $R$ is a (graded) Poisson algebra, and $M$ is a (graded) $R$-module, then we can verify that $JR \otimes_R M$ is a (graded) $JR$-module by defining the Poisson structure as
\begin{equation*}
  a_{(n)}(b\otimes u) = (a_{(n)}b)\otimes u + \delta_{n, 0}b\otimes\{a, u\} \quad \text{for $a \in R$, $b \in JR$, $u \in M$ and $n \in \mathbb{N}$}
\end{equation*}
and $H^{JR \otimes_R M}$ as in \zcref{sec:graded-jet-algebras}.

The (graded) $R$-module $M$ together with the natural (graded) $R$-module inclusion $\inc: M \hookrightarrow JR \otimes_R M$ satisfy the following universal property.
Let $N$ be a (graded) module over the (graded) vertex Poisson algebra $JR$, and let $\phi: M \to N$ be a (graded) homomorphism of modules over the (graded) Poisson algebra $R$.
There is a unique (graded) homomorphism $\overline{\phi}: JR \otimes_R M \to N$ of (graded) modules over the (graded) vertex Poisson algebra $JR$ such that the following diagram commutes
\begin{equation*}
  \begin{tikzcd}
    M \arrow[r, "\inc", hook] \arrow[rd, "\phi"'] & JR \otimes_R M \arrow[d, "\overline{\phi}"] \\
    & N
  \end{tikzcd}
\end{equation*}

\bibliographystyle{alpha}
\bibliography{boundary-minimal-models}

\end{document}